\documentclass[11pt,a4paper]{amsart}

\usepackage[T1]{fontenc}
\usepackage[utf8]{inputenc}

\usepackage[english]{babel}
\usepackage{amsmath,amsfonts,amssymb,amsthm}

\usepackage[pdftex]{graphicx}
\usepackage{lmodern}

\usepackage[bottom=3cm, top=3cm, left=3.5cm, right=3.5cm]{geometry}

\usepackage[foot]{amsaddr}

\usepackage{mathtools}
\mathtoolsset{showonlyrefs,showmanualtags}	

\DeclareMathAlphabet{\mathbbold}{U}{bbold}{m}{n}	

\theoremstyle{plain}
\newtheorem{theorem}{Theorem}
\newtheorem*{theorem*}{Theorem}
\newtheorem{prop}[theorem]{Proposition}
\newtheorem{cor}[theorem]{Corollary}

\newtheorem{lemma}[theorem]{Lemma}
\newtheorem*{lemma*}{Lemma}
\theoremstyle{definition}
\newtheorem{definition}[theorem]{Definition}
\theoremstyle{remark}

\newtheorem{remark}[theorem]{Remark}
\newtheorem{rmk}[theorem]{Remark}

\usepackage{hyperref}
\usepackage{xcolor}
\hypersetup{
    colorlinks,
    linkcolor={red!50!black},
    citecolor={blue!50!black},
    urlcolor={blue!80!black}
}

\usepackage[alphabetic, abbrev]{amsrefs}

\newcommand{\mis}{\mathsf{m}}	
\newcommand{\G}{\mathbb{G}}
\newcommand{\R}{\mathbb{R}}
\newcommand{\N}{\mathbb{N}}
\newcommand{\distr}{\mathcal{D}}

\newcommand{\eps}{\varepsilon}
\newcommand{\lam}{\lambda}
\newcommand{\J}{\mathcal{J}}
\newcommand{\g}{\gamma}

\newcommand{\cutopt}{\mathrm{CutOpt}}
\newcommand{\cut}{\mathrm{Cut}}

\newcommand{\SC}{\mathrm{SC}}	
\newcommand{\Abn}{\mathrm{Abn}}	

\newcommand{\ver}{\mathcal{V}}
\newcommand{\hors}{\textsc{h}}
\newcommand{\vers}{\textsc{v}}

\DeclareMathOperator{\spn}{\mathrm{span}}
\DeclareMathOperator{\rank}{\mathrm{rank}}

\DeclareMathOperator{\supp}{\mathrm{supp}}
\DeclareMathOperator{\Hess}{\mathrm{Hess}}

\usepackage{todonotes}

\let\oldtocsection=\tocsection
 
\let\oldtocsubsection=\tocsubsection

\let\oldtocsubsubsection=\tocsubsubsection
 
\renewcommand{\tocsection}[2]{\hspace{0em}\oldtocsection{#1}{#2}}
\renewcommand{\tocsubsection}[2]{\hspace{1em}\oldtocsubsection{#1}{#2}}
\renewcommand{\tocsubsubsection}[2]{\hspace{2em}\oldtocsubsubsection{#1}{#2}}

\usepackage{enumitem}
\setlist[itemize]{leftmargin=1cm}

\author[Davide Barilari]{Davide Barilari$^\flat$}
\address{$^\flat$ Institut de Math\'ematiques de Jussieu-Paris Rive Gauche, UMR CNRS 7586, Universit\'e  Paris-Diderot, Batiment Sophie Germain, Case 7012, 75205 Paris Cedex 13, France}
\email{\href{mailto:davide.barilari@imj-prg.fr}{davide.barilari@imj-prg.fr}}

\author[Luca Rizzi]{Luca Rizzi$^\sharp$}
\address{$^\sharp$ Univ. Grenoble Alpes, CNRS, Institut Fourier, 38000 Grenoble, France}
\email{\href{mailto:luca.rizzi@univ-grenoble-alpes.fr}{luca.rizzi@univ-grenoble-alpes.fr}}

\allowdisplaybreaks

\title[Sub-Riemannian interpolation inequalities]{Sub-Riemannian interpolation inequalities}


\begin{document}

\begin{abstract}
We prove that ideal sub-Rie\-man\-nian manifolds (i.e., admitting no non-trivial abnormal minimizers) support interpolation inequalities for optimal transport.
A key role is played by sub-Riemannian Jacobi fields and distortion coefficients, whose properties are remarkably different with respect to the Riemannian case. As a byproduct, we characterize the cut locus as the set of points where the squared sub-Riemannian distance fails to be semiconvex, answering to a question raised by Figalli and Rifford in \cite{FR-mass}.

As an application, we deduce sharp and intrinsic Borell-Brascamp-Lieb and geodesic Brunn-Minkowski inequalities in the aforementioned setting. For the case of the Heisenberg group, we recover in an intrinsic way the results recently obtained by Balogh, Krist\'aly and Sipos in \cite{BKS-geomheis}, and we extend them to the class of generalized H-type Carnot groups. Our results do not require the distribution to have constant rank, yielding for the particular case of the Grushin plane a sharp measure contraction property and a sharp Brunn-Minkowski inequality.
\end{abstract}

\maketitle

\tableofcontents

\section{Introduction}\label{s:intro}

In the seminal paper \cite{CEMS-interpolations} it is proved that some natural inequalities holding in the Euclidean space generalize to the Riemannian setting, provided that the geometry of the ambient space is taken into account through appropriate distortion coefficients. The prototype of these inequalities in $\R^{n}$ is the Brunn-Minkowski one, or its functional counterpart in the form of Borell-Brascamp-Lieb inequality.

The main results of \cite{CEMS-interpolations}, which are purely geometrical, were originally formulated in terms of optimal transport. The theory of optimal transport (with quadratic cost) is nowadays well understood in the Riemannian setting, thanks to the works of McCann \cite{McCann-polar}, who adapted to manifolds the theory of Brenier in the Euclidean space \cite{Brenier99}. We refer to \cite{V-oldandnew} for references, including a complete historical account of the theory and its subsequent developments.

Let then $\mu_0$ and $\mu_1$ be two probability measures on an $n$-dimensional Riemannian manifold $(M,g)$. We assume $\mu_0,\mu_1$ to be compactly supported, and absolutely continuous with respect to the Riemannian measure $\mis_g$, so that $\mu_i = \rho_i \mis_g$ for some $\rho_i \in L^1(M,\mis_g)$. Under these assumptions, there exists a unique optimal transport map $T : M \to M$, such that $T_\sharp \mu_0 = \mu_1$ and which solves the Monge problem:
\begin{equation}
\int_M d^2(x,T(x)) d\mis_g(x) = \inf_{S_\sharp \mu_0 = \mu_1} \int_M d^2(x,S(x)) d\mis_g(x).
\end{equation}
Furthermore, for $\mu_0-$a.e.\ $x \in M$, there exists a unique constant-speed geodesic $T_t(x)$, with $0 \leq t \leq 1$, such that $T_0(x) = x$ and $T_1(x) = T(x)$. The map $T_t :M \to M$ defines the dynamical interpolation $\mu_t = (T_t)_\sharp \mu_0$, a curve in the space of probability measures joining $\mu_0$ with $\mu_1$. 
More precisely, $(\mu_t)_{0\leq t \leq 1}$ is the unique Wasserstein geodesic between $\mu_0$ and $\mu_1$, with respect to the quadratic transportation cost.

By a well-known regularity result, $\mu_t$ is absolutely continuous with respect to $\mis_g$, that is $\mu_t = \rho_t \mis_g$ for some $\rho_t \in L^1(M,\mis_g)$. The fundamental result of \cite{CEMS-interpolations} is that the concentration $1/\rho_t$ during the transportation process can be estimated with respect to its initial and final values. More precisely, for all $t \in [0,1]$, the following \emph{interpolation inequality} holds:
\begin{equation}\label{eq:CD-intro}
\frac{1}{\rho_t(T_t(x))^{1/n}} \geq \frac{\beta_{1-t}(T(x),x)^{1/n}}{\rho_0(x)^{1/n}} + \frac{\beta_t(x,T(x))^{1/n}}{\rho_1(T(x))^{1/n}}, \qquad \mu_0-\mathrm{a.e.}\, x \in M.
\end{equation}
Here, $\beta_t(x,y)$, for $t \in [0,1]$, are \emph{distortion coefficients} which depend only on the geometry of the underlying Riemannian manifold, and can be computed once the Riemannian structure is given,  see Definition~\ref{d:introdist}.  The distortion coefficients are in general  difficult to compute but, if $\mathrm{Ric}_g(M) \geq K g$, then $\beta_t(x,y)$ are controlled from below by their analogues on the Riemannian space forms of constant curvature equal to $K$ and dimension $n$. More precisely, we have
\begin{equation}\label{eq:dist-ref-riem-intro}
\beta_t(x,y) \geq \beta_t^{(K,n)}(x,y) = \begin{cases}
t \left(\frac{\sin(t \alpha)}{\sin( \alpha)}\right)^{n-1} & \text{if } K>0, \\ 
t^n & \text{if } K=0 , \\
t \left(\frac{\sinh(t  \alpha)}{\sinh(  \alpha)}\right)^{n-1} & \text{if } K<0,
\end{cases}
\end{equation}
where
\begin{equation}
\alpha = \sqrt{\frac{|K|}{n-1}}d(x,y).
\end{equation}

Inequality \eqref{eq:CD-intro}, when expressed in terms of the reference coefficients \eqref{eq:dist-ref-riem-intro}, is one of the incarnations of the so-called curvature-dimensions $\mathrm{CD}(K,N)$ condition, which allows to generalize the concept of Ricci curvature bounded from below and dimension bounded from above to more general metric measure spaces. This is the beginning of the synthetic approach propugnated by Lott-Villani and Sturm \cite{LV-ricci,S-ActaI,S-ActaII} and extensively developed subsequently.

The main tools used in \cite{CEMS-interpolations} are Jacobi fields and the properties of the cut locus, the nature of which changes dramatically in the sub-Riemannian setting. For this reason the extension of the above inequalities to the sub-Riemannian world has remained elusive (see also \cite{O-Finslerinterpolation} for a discussion of the Finsler case). For example, it is now well-known that the Heisenberg group equipped with a left-invariant measure, which is the simplest sub-Riemannian structure, does not satisfy any form of $\mathrm{CD}(K,N)$, as proved in \cite{Juillet}. 

On the other hand, it has been recently proved in \cite{BKS-geomheis} that the Heisenberg group actually supports interpolation inequalities as \eqref{eq:CD-intro}, with distortion coefficients whose properties are quite different with respect to the Riemannian case. The techniques in \cite{BKS-geomheis} consist in employing a one-parameter family of Riemannian extension of the Heisenberg structure, converging to the latter as $\varepsilon \to 0$. Starting from the Riemannian interpolation inequalities, a fine analysis is required to obtain a meaningful limit for $\varepsilon \to 0$. It is important to stress that the Ricci curvature of the Riemannian extensions is unbounded from below as $\varepsilon \to 0$. 

The results of \cite{BKS-geomheis} and the extension to the corank $1$ case of \cite{BKS-geomcorank1} suggest that a sub-Riemannian  theory of interpolation inequalities which parallels the Riemannian one actually exists. In this paper, we answer to the following question:
\begin{center}
\emph{Do sub-Riemannian manifolds support weighted interpolation inequalities \`a la \cite{CEMS-interpolations}? How to recover the correct weights and what are their properties?}
\end{center}
We obtain a satisfying and positive answer, at least for the so-called \emph{ideal} structures, that is admitting no non-trivial abnormal minimizing geodesics (this is a generic property, see Proposition~\ref{p:generic}). This is the most general setting in which the sub-Riemannian transportation problem is well posed (see Section~\ref{s:srtransport}).

\subsection{Interpolation inequalities}
To introduce our results, let $(\distr,g)$ be a sub-Riemannian structure on a smooth manifold $M$, and fix a smooth reference (outer) measure $\mis$. Let us introduce the (sub-)Riemannian distortion coefficients.

\begin{definition}\label{d:introdist0}
Let $A,B \subset M$ be measurable sets, and $t \in [0,1]$. The set $Z_t(A,B)$ of $t$-\emph{intermediate points} is the set of all points $\gamma(t)$, where $\gamma:[0,1] \to M$ is a minimizing geodesic such that $\gamma(0) \in A$ and $\gamma(1) \in B$.
\end{definition}
\begin{figure}[t]
\includegraphics[scale=1]{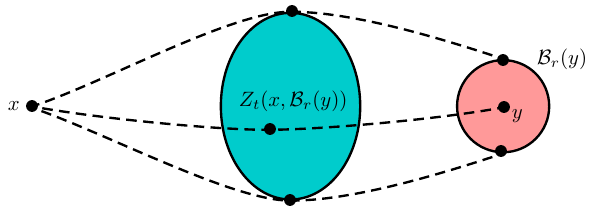}
\caption{The distortion coefficient $\beta_t(x,y)$.}
\label{f:figbm1}
\end{figure}
Let $\mathcal{B}_r(x)$ denote the sub-Riemannian ball of center $x\in M$ and radius $r>0$.
\begin{definition}[Distortion coefficient]\label{d:introdist}
Let $x, y \in M$. The \emph{distortion coefficient} from $x$ to $y$ at time $t \in [0,1]$ is
\begin{equation}\label{eq:distcoeff}
\beta_t(x,y) := \limsup_{r \downarrow 0} \frac{\mis(Z_t(x,\mathcal{B}_r(y)))}{\mis(\mathcal{B}_r(y))}.
\end{equation}
Notice that $\beta_0(x,y) = 0$ and $\beta_1(x,y) = 1$.
\end{definition}
\begin{rmk}
In the Riemannian case $\beta_t(x,y) \sim t^n$ for $t\to 0$. This universal asymptotics, valid in the Riemannian case, led \cite{CEMS-interpolations} to extract a factor $t^{n}$ in \eqref{eq:CD-intro}, expressing it in terms of the modified distortion coefficients $v_{t}(x,y):= \beta_{t}(x,y)/t^{n}$. The main difference is that here we do not extract a factor $1/t^n$, since the topological dimension does not describe the correct asymptotic behavior in the sub-Riemannian case (see Theorem~\ref{t:asymptotics-intro}). Compare also \eqref{eq:distcoeff} with \cite[Def.\ 14.17, Prop.\ 14.18]{V-oldandnew}.
\end{rmk}

Despite the lack of a canonical Levi-Civita connection and curvature, in this paper we develop a suitable theory of sub-Riemannian (or rather Hamiltonian) Jacobi fields, which is powerful enough to derive interpolation inequalities. Our techniques are based on the  approach initiated in \cite{agzel1,agzel2,lizel}, and subsequently developed in a language that is  closer to our presentation in \cite{curvature,BR-connection,BR-comparison}. Our first main result is the extension of \eqref{eq:CD-intro} to the ideal sub-Riemannian setting.
\begin{theorem}[Interpolation inequality] \label{t:jacointerpineq}
Let $(\distr,g)$ be an ideal sub-Riemannian structure on $M$, and $\mu_0,\mu_1 \in \mathcal{P}_c^{ac}(M)$. Let $\rho_s = d \mu_s / d\mis$. For all $t \in [0,1]$, it holds
\begin{equation}\label{eq:jacinterpineq-intro}
\frac{1}{\rho_t(T_t(x))^{1/n}} \geq \frac{\beta_{1-t}(T(x),x)^{1/n}}{\rho_0(x)^{1/n}} + \frac{\beta_t(x,T(x))^{1/n}}{\rho_1(T(x))^{1/n}}, \qquad \mu_0-\mathrm{a.e.}\, x \in M.
\end{equation}
If $\mu_1$ is not absolutely continuous, an analogous result holds, provided that $t \in [0,1)$, and that in \eqref{eq:jacinterpineq-intro} the second term on the right hand side is omitted.
\end{theorem}

Theorem \ref{t:jacointerpineq} is proved in Section~\ref{s:iipt4}. A key role in the proof is played by a positivity lemma (cf.\ Lemma \ref{l:keylemma1}) inspired by \cite[Ch.\ 14, Appendix: Jacobi fields forever]{V-oldandnew}, which allows to overcome the non positive definiteness of the sub-Riemannian Hamiltonian. Moreover, with respect to previous approaches, we stress that we do not make use of any canonical frame, playing the role of a parallel transported frame. 

Concerning the sub-Riemannian distortion coefficients, they can be explicitly computed in terms of sub-Riemannian Jacobi fields (cf.\ Lemma~\ref{l:distortioncomputed}). This relation is then used in Section~\ref{s:esempi} to yield explicit formulas in different examples.


Even in the most basic examples, the distortion coefficients have some peculiar features that have no analogue in the Riemannian case. These are discussed in Section~\ref{s:distortionproperties}. Here we only give the following statement, which also allows us to introduce the important concept of geodesic dimension (for a proof see Section~\ref{s:distortionproperties}).

\begin{theorem}[Asymptotics of sub-Riemannian  distortion]\label{t:asymptotics-intro}
Let $(\distr,g)$ be a sub-Riemannian structure on $M$, not necessarily ideal. Let $x \in M$ and $y \notin \cut(x)$. Then, there exists an integer $\mathcal{N}(x,y)$ and a constant $C(x,y)>0$ such that
\begin{equation}
\beta_t(x,y) \sim C(x,y) t^{\mathcal{N}(x,y)}, \qquad \text{for } t \to 0^+.
\end{equation}
Furthermore, for a.e.\ $y \notin \cut(x)$, the exponent $\mathcal{N}(x,y)$ attains its minimal value
\begin{equation}
\mathcal{N}(x):= \min\{\mathcal{N}(x,z) \mid z \notin \cut(x)\}.
\end{equation}
The number $\mathcal{N}(x)$ is called the \emph{geodesic dimension} of the sub-Riemannian structure at $x$. Finally, the following inequality holds
\begin{equation}
\mathcal{N}(x) \geq \dim(M),
\end{equation}
with equality if and only if the structure is Riemannian at $x$, that is $\distr_{x} = T_{x} M$.
\end{theorem}
We mention that there is an explicit formula for the geodesic dimension of a sub-Riemannian manifold of the form
\begin{equation}
\mathcal{N}(x)=\sum_{i=1}^{m}(2i-1)(\dim \mathcal{F}_{x}^{i}-\dim \mathcal{F}_{x}^{i-1}),
\end{equation}
where $\mathcal{F}_{x}^{1}\subset\cdots\subset \mathcal{F}_{x}^{m} = T_{x}M$ is a flag of subspaces  associated to  generic geodesics. This formula is reminiscent of Mitchell's formula \cite{Mitchell} for Hausdorff dimension for equiregular manifolds
\begin{equation}
Q=\sum_{j=1}^{r}j(\dim \distr_{x}^{j}-\dim \distr_{x}^{j-1}),
\end{equation}
where $\distr_{x}^{1}\subset\cdots\subset \distr_{x}^{r} = T_{x}M$ is the classical flag of  the distribution.
We stress that the two flags are different, in general. The geodesic dimension was initially discovered for sub-Riemannian structures in \cite[Sec.\ 5.6]{curvature}, and generalized to metric measure spaces in \cite{R-MCPcorank1}, to which we refer for more details.

\subsection{Regularity of distance}

The proof of Theorem~\ref{t:jacointerpineq} is also related with the structure of the cut locus.
In Riemannian geometry, it is well-\\
known that for almost every geodesic $\gamma$ involved in the transport, $\gamma(1) \notin \cut(\gamma(0))$. In particular, this implies (in a non-trivial way), that the cut locus, which is defined as the set of points where the squared distance is not smooth, can be characterized actually as the set of points where the squared distance fails to be semiconvex \cite{CEMS-interpolations}. 

Here, we extend the latter to the sub-Riemannian setting, answering affirmatively to the open problem raised by Figalli and Rifford in \cite[Sec.\ 5.8]{FR-mass}, at least when non-trivial abnormal geodesics are not present.

\begin{theorem}[Failure of semiconvexity at the cut locus]\label{t:FRc-intro}
Let $(\distr,g)$ be an ideal sub-Riemannian structure on $M$. Let $y \neq x$. Then $x \in \cut(y)$ if and only if the squared sub-Riemannian distance from $y$ fails to be semiconvex at $x$, that is, in local coordinates around $x$, we have
\begin{equation}
\inf_{0<|v|<1} \frac{d^2_{SR}(x+v,y) + d^2_{SR}(x-v,y) - 2d^2_{SR}(x,y)}{|v|^2} = - \infty.
\end{equation}
\end{theorem}

The characterization of Theorem~\ref{t:FRc-intro} is false in the non-ideal case, as we discuss in Section~\ref{s:regopt}. Some related open problems are proposed in Section~\ref{s:openquestions}.

\subsection{Geometric inequalities}

The classical consequences of interpolation inequalities  follow from standard arguments. In Section~\ref{s:inequalities} we discuss the the $p$-mean and the Borell-Brascamp-Lieb inequalities (Theorems \ref{t:srpmean} and \ref{t:srbbl}, respectively). In this introduction we focus on their geometric counterpart: the Brunn-Minkowski inequality.
Its classical version asserts that  for measurable sets $A,B\subset \R^{n}$   one has
\begin{equation}
|(1-  t)A + tB|^{1/n} \geq 
 (1 -  t)|A|^{1/n}
+ t |B|^{1/n},\qquad  0 \leq  t \leq 1, 
\end{equation}
 where $|\cdot|$ denotes the Lebesgue measure.
 The set $Z_{t}(A, B) = (1- t)A + tB$ 
consists of the locus of points $\gamma(t)$ as $\gamma$ varies over all line segments $(1-t)x+ty$, for  $0 \leq t \leq 1$ joining points $x \in A$ to points $y\in  B$. 
We refer to \cite{Gardner} for a comprehensive review in the Euclidean context.

%


To introduce the geodesic Brunn-Minkowski inequality, we define for any pair of Borel subsets $A,B\subset M$ the  following quantity:
\begin{equation} \label{eq:betatAB}
\beta_{t}(A,B):=\inf\left\{ \beta_{t}(x,y)  \mid (x,y)\in (A\times B)\setminus \cut(M)\right\},
\end{equation}
with the convention that $\inf \emptyset =0$. Notice that $0\leq \beta_{t}(A,B)<+\infty$, cf.\ Lemma \ref{l:distortioncomputed}.

\begin{theorem}[Sub-Riemannian Brunn-Minkowski inequality] \label{t:bmsr-intro}
Let $(\distr,g)$ be an ideal sub-Rieman\-nian structure on a $n$-dimensional manifold $M$, equipped with a smooth measure $\mis$. Let $A,B\subset M$ be Borel subsets.  Then we have
\begin{equation}\label{eq:bmsr-intro}
\mis(Z_{t}(A,B))^{1/n}\geq \beta_{1-t}(B,A)^{1/n} \mis(A)^{1/n}+ \beta_{t}(A,B)^{1/n} \mis(B)^{1/n}.
\end{equation}
\end{theorem}\begin{figure}[!ht]
\includegraphics[scale=1]{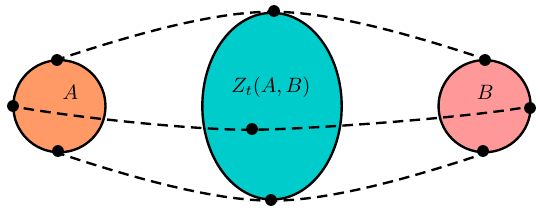}
\caption{The set $Z_{t}(A,B)$.}
\label{f:figbm2}
\end{figure}

\begin{remark}
A different generalization of the Euclidean Brunn-Minkowski inequality, at least for left-invariant structures on Lie groups, is the \emph{multiplicative} Brunn-Minkowski inequality. The latter is defined by replacing the Minkowski sum $A+B$ with the group multiplication $A \star B$. For the Heisenberg group $\mathbb{H}_3$, with group law $\star$ and left-invariant measure $\mis$, the multiplicative Brunn-Minkowski inequality reads
\begin{equation}
\mis(A \star B)^{1/d} \geq \mis(A)^{1/d} + \mis(B)^{1/d}, \qquad A,B \subset \mathbb{H}^3.
\end{equation}
The above inequality is true for the topological dimension $d=3$ \cite{LM-isopH}, but false for the Hausdorff dimension $d=4$ \cite{M-BMisoH}.
\end{remark}

A particular role in Theorem~\ref{t:bmsr-intro} is played by structures where $\beta_t(x,y) \geq t^N$ for some $N \in \N$, for all $t \in [0,1]$ and $(x,y) \notin \cut(M)$. By Theorem \ref{t:bmsr-intro}, this implies the so-called \emph{measure contraction property} $\mathrm{MCP}(0,N)$, first introduced in \cite{Ohta} (see also \cite{S-ActaII} for a similar formulation). The $\mathrm{MCP}$ was first investigated in Carnot groups in \cite{Juillet,Riff-Carnot}. In our setting, we prove the following equivalence result.

\begin{theorem}[Equivalence of inequalities]\label{t:equivalenza}
Let $(\distr,g)$ be an ideal sub-Rieman\-nian structure on a $n$-dimensional manifold $M$, equipped with a smooth measure $\mis$. Let $N\geq 1$.
Then, the following properties are equivalent:
\begin{itemize} 
\item[(i)] $\beta_t(x,y) \geq t^N$, for all $(x,y) \notin \cut(M)$ and  $t \in [0,1]$;
\item[(ii)] the Brunn-Minkowski inequality holds: for all non-empty Borel sets $A,B$
\begin{equation}\label{eq:bmsr-intro0}
\mis(Z_{t}(A,B))^{1/n}\geq (1-t)^{N/n} \mis(A)^{1/n}+ t^{N/n} \mis(B)^{1/n}, \qquad \forall\, t \in [0,1];
\end{equation}
\item[(iii)] the measure contraction property $\mathrm{MCP}(0,N)$ is satisfied: for all non-empty Borel sets $B$ and $x\in M$
\begin{equation}\label{eq:bmsr-intro00}
\mis(Z_{t}(x,B))\geq  t^{N} \mis(B), \qquad \forall\, t \in [0,1].
\end{equation}
\end{itemize}
\end{theorem}
We stress that on a $n$-dimensional sub-Riemannian manifold that is not Riemannian, the $\mathrm{MCP}(0,n)$ is never satisfied (see \cite[Thm.\ 6]{R-MCPcorank1}). 

This clarifies the fact that an Euclidean Brunn-Minkowski inequality with linear weights (that is \eqref{eq:bmsr-intro0} with  $N=n$), is not adapted for generalizations to genuine sub-Riemannian situations, as well as the classical curvature-dimension condition. We mention that generalized curvature-dimension type inequalities suitable for particular classes of sub-Riemannian structures have been developed in \cite{BG-CD,BKW-transverse}.

\subsection{Old and new examples}
In Section~\ref{s:esempi}, we discuss some examples, where the distortion coefficients are explicit. In particular, we consider:
\begin{itemize}
\item \textbf{The Heisenberg group $\mathbb{H}_3$}. In this case we recover, in an intrinsic way, the results of \cite{BKS-geomheis}, with the same distortion coefficients. See Section~\ref{s:heis}.
\item \textbf{Generalized H-type groups}. This is a class of Carnot groups of arbitrary large corank, introduced in \cite{BR-MCPHtype}, and which extends the class of Kaplan H-type groups. In the ideal case we obtain sharp interpolations inequalities for general measures (Corollary~\ref{c:interpohtype}). In the general and possibly non-ideal case, we prove sharp Brunn-Minkowski inequalities (Corollary~\ref{c:BMhtype}) and measure contraction properties. 
See Section~\ref{s:genheis}.
\item \textbf{Grushin plane $\mathbb{G}_2$}. Our techniques work also for sub-Riemannian distributions $\distr$ whose rank is not constant. In this setting we are able to obtain for the first time interpolation inequalities (Corollary~\ref{c:grushin-interpolation}), sharp Brunn-Minkowski inequalities (Corollary~\ref{c:grushin-BM}), and sharp measure-contraction properties (Corollary~\ref{c:grushin-MCP}). See Section~\ref{s:grush}.
\item \textbf{Sasakian structures}. Sasakian manifolds are a particular class of contact sub-Riemannian structures. When endowed with their canonical volume, Sasakian manifolds satisfy a measure contraction property under suitable curvature lower bounds. Combining these results with Theorem~\ref{t:equivalenza}, we get a sharp Brunn-Minkowski inequality (Corollary~\ref{c:sasaki}). See Section~\ref{s:sasaki}.

\end{itemize}

In all the above cases, we are able to prove that the distortion coefficients satisfy
\begin{equation} \label{eq:ineqgeneral-intro}
\beta_t(x,y) \geq t^\mathcal{N}, \qquad \forall (x,y) \notin \cut(M), \, \forall t \in [0,1],
\end{equation}
for some minimal $\mathcal{N}$, given by the geodesic dimension of the sub-Rie\-mannian structure (cf.\ Theorem~\ref{t:asymptotics-intro}). 
The interpolation inequalities take hence a very pleasant sharp form. For example in the case of the Brunn-Minkowski inequality, for all non-empty Borel sets $A,B$, we have 
\begin{equation}\label{eq:ineqgeneral-intro3}
\mis(Z_t(A,B))^{1/n} \geq (1-t)^{\mathcal{N}/n} \mis(A)^{1/n} + t^{\mathcal{N}/n} \mis(B)^{1/n}, \qquad \forall\, t \in [0,1],
\end{equation}
In all these cases, \eqref{eq:ineqgeneral-intro3} is sharp, in the sense that if one replaces the exponent $\mathcal{N}$ with a smaller one, the inequality fails for some choice of $A,B$.
%

\subsection{Carnot groups}
Another class of examples is given by ideal Carnot groups. In \cite{Riff-Carnot} it was proved that for any ideal Carnot group there exists $N \geq \mathcal{N}$ such that the $\mathrm{MCP}(0,N)$ property is satisfied (see \cite{Rif-mediumfat} for the generalization to the medium-fat case). Hence we have the following result.
\begin{cor}
For any ideal Carnot group $G$ there exists $N  \geq \mathcal{N}$ such that $G$, equipped with its left-invariant sub-Riemannian structure and the Haar measure, satisfies the inequalities of Theorem~\ref{t:equivalenza}.
\end{cor}
\noindent
\textbf{Open questions.} Let $G$ be a Carnot group equipped with a left-invariant sub-Riemannian structure and the Haar measure.  
\begin{itemize}
\item[(i)] Is it true that the Brunn-Minkowski type inequality \eqref{eq:bmsr-intro0} holds for some $N\in \N$? 
\item[(ii)] Is the optimal $N$ such that \eqref{eq:bmsr-intro0} holds equal to the geodesic dimension?
\end{itemize}
In question (i), if such $N$ exists, then it is greater or equal than the geodesic dimension $\mathcal{N}$ of the Carnot group, as a consequence of the asymptotics of Theorem~\ref{t:asymptotics-intro} as $t\to 0$. Indeed, a proof of the Brunn-Minkowski inequality would require a control on the distortion coefficients for all $t\in [0,1]$. 

\subsection{Afterwords}

In this work we focused in laying the groundwork for interpolation inequalities in sub-Riemannian geometry. It remains to understand which is the correct class of models whose distortion coefficients are the reference ones, playing the role of Riemannian space forms in Riemannian geometry. This will be the object of a subsequent work. We anticipate here that the natural reference spaces \emph{do not belong} to the category of sub-Riemannian structures. The unifying framework that we propose is the one of optimal control problems. This class of variational problems is  large enough to include infinitesimal models for all of the three great classes of geometries: Riemannian, sub-Riemannian, Finslerian, providing the first step of the ``great unification'' auspicated in \cite[Sec.\ 9]{Villani-Bourbaki}. In the spirit of \cite{BR-comparison}, linear-quadratic optimal control problems play the role of constant curvature spaces.

Another challenging problem is to understand how to include abnormal minimizers in this picture. Abnormal geodesics, as \cite{BKS-geomcorank1} suggests for the case of corank $1$ Carnot groups, are not \emph{a priori} an obstacle to interpolation inequalities. These remarkable results are the consequence of the special structure of corank $1$ Carnot groups, which are the metric product of an (ideal) contact Carnot group and a suitable copy of a flat $\R^n$. In general, an organic theory of transport and Jacobi fields along abnormal geodesics is still lacking. In this paper, we discuss some aspects of the non-ideal case and some open problems in Section~\ref{s:regopt}.			
\section{Preliminaries}\label{s:prel}

We start by recalling some basic facts in sub-Riemannian geometry. For a comprehensive introduction, we refer to \cite{nostrolibro,riffordbook,montgomerybook}.

\subsection{Sub-Riemannian geometry}
A sub-Rieman\-nian structure on a smooth, connected $n$-dimensional manifold $M$, where $n\geq 3$, is defined by a set of $m$ global smooth vector fields $X_{1},\ldots,X_{m}$, called a \emph{generating frame}. The \emph{distribution} is the family of subspaces of the tangent spaces spanned by the vector fields at each point
\begin{equation}
\distr_{x}=\mathrm{span}\{X_{1}(x),\ldots,X_{m}(x)\}\subseteq T_{x}M,\qquad \forall\, x\in M.
\end{equation}
The generating frame induces an inner product $g_{x}$ on $\distr_{x}$ as follows: given $v,w\in \distr_x$ the inner product $g_{x}(v,w)$ is defined by the polarization formula, letting
\begin{equation}
g_{x}(v,v):=\inf\left\{\sum_{i=1}^{m}u_{i}^{2}\mid \sum_{i=1}^{m}u_{i}X_{i}(x)=v\right\}.
\end{equation}
We assume that the distribution is \emph{bracket-generating}, i.e., the tangent space $T_{x}M$ is spanned by the vector fields $X_{1},\ldots,X_{m}$ and their iterated Lie brackets evaluated at $x$. A \emph{horizontal curve} $\gamma : [0,1] \to M$ is an absolutely continuous path such that there exists $u\in L^{2}([0,1],
\R^{m})$ satisfying
\begin{equation}
\dot\gamma(t) =  \sum_{i=1}^m u_i(t) X_i(\gamma(t)), \qquad \mathrm{a.e.}\, t \in [0,1].
\end{equation}
This implies that $\dot\gamma(t) \in \distr_{\gamma(t)}$ for almost every $t$. If $\gamma$ is horizontal, the map $t\mapsto \sqrt{g(\dot\gamma(t),\dot\gamma(t))}$ is measurable on $[0,1]$, hence integrable \cite[Lemma 3.11]{ABBEMS}. We define the \emph{length} of an horizontal curve as follows
\begin{equation}
\ell(\gamma) = \int_0^1 \sqrt{g(\dot\gamma(t),\dot\gamma(t))}dt.
\end{equation}
The \emph{sub-Rieman\-nian distance} is defined by:
\begin{equation}\label{eq:infimo}
d_{SR}(x,y) = \inf\{\ell(\gamma)\mid \gamma(0) = x,\, \gamma(1) = y,\, \gamma \text{ horizontal} \}.
\end{equation}

\begin{remark} The above definition includes rank-varying sub-Riemannian structures on $M$, see \cite{ABBEMS}. When $\dim \distr_{x}$ is constant, then $\distr$ is a vector distribution in the classical sense. If $m\leq n$ and the vector fields $X_{1},\ldots,X_{m}$ are linearly independent everywhere, they form a basis of $\distr$ and $g$ coincides with the inner product on $\distr$ for which $X_{1},\ldots,X_{m}$ is an orthonormal frame.
\end{remark}
By Chow-Rashevskii theorem, the bracket-generating condition implies that $d_{SR}: M \times M \to \R$ is finite and continuous. If the metric space $(M,d_{SR})$ is complete, then for any $x,y \in M$ the infimum in \eqref{eq:infimo} is attained. In place of the length $\ell$, it is often convenient to consider the \emph{energy functional} 
\begin{equation}
J(\gamma) = \frac{1}{2}\int_0^1 g(\dot\gamma(t),\dot\gamma(t)) dt.
\end{equation}
On the space of horizontal curves defined on a fixed interval and with fixed endpoints, the minimizers of $J$ coincide with the minimizers of $\ell$ parametrized with constant speed. Since $\ell$ is invariant by reparametrization, and every horizontal curve is the reparametrization of a constant-speed one, we define \emph{geodesics} as horizontal curves that locally minimize the energy between their endpoints.

The \emph{Hamiltonian} of the sub-Riemannian structure $H : T^*M \to \R$ is defined by
\begin{equation}
H(\lambda) = \frac{1}{2}\sum_{i=1}^{m} \langle \lambda,X_i\rangle^{2}, \qquad \lambda \in T^*M,
\end{equation}
where $X_1,\dots,X_m$ is the generating frame. Here $\langle \lambda,\cdot\rangle$ denotes the dual action of covectors on vectors. Different generating frames defining the same distribution and scalar product at each point yield the same Hamiltonian. The Hamiltonian vector field $\vec{H}$ is the unique vector field such that $\sigma(\cdot,\vec{H}) = dH$, where $\sigma$ is the canonical symplectic form on $T^*M$. In particular, the \emph{Hamilton equations} are
\begin{equation}\label{eq:Ham}
\dot\lambda(t) = \vec{H}(\lambda(t)), \qquad \lambda(t) \in T^*M.
\end{equation}
If $(M,d_{SR})$ is complete, solutions of \eqref{eq:Ham} are defined for all times.

\subsection{End-point map and Lagrange multipliers}\label{s:eplm}

Let $\gamma_u :[0,1] \to M$ be an horizontal curve joining $x$ and $y$, where $u \in L^2([0,1],\R^m)$ is a \emph{control} such that
\begin{equation}
\dot\gamma_u(t) =  \sum_{i=1}^m u_i(t) X_i(\gamma_u(t)), \qquad \text{a.e. } t \in [0,1].
\end{equation}
Let $\mathcal{U} \subset L^2([0,1],\R^m)$ be the neighbourhood of $u$ such that, for $v \in \mathcal{U}$, the equation
\begin{equation}
\dot\gamma_v(t) = \sum_{i=1}^m v_i(t) X_i(\gamma_v(t)), \qquad \gamma_v(0) = x,
\end{equation}
has a well defined solution for a.e. $t \in [0,1]$. We define the \emph{end-point map} with base point $x$ as the smooth map $E_{x}: \mathcal{U} \to M$, which sends $v$ to $\gamma_v(1)$.

We can consider $J : \mathcal{U} \to \R$ as a smooth functional on $\mathcal{U}$. Let $\gamma_u$ be a minimizing geodesic, that is a solution of the constrained minimum problem
\begin{equation}
\min\{J(v) \mid  v \in \mathcal{U},\, E_x(v) = y\}.
\end{equation}
By the Lagrange multipliers rule, there exists a non-trivial pair $(\lambda_1,\nu)$, such that
\begin{equation}\label{eq:multipliers}
\lambda_1 \circ D_u E_x  = \nu D_u J, \qquad \lambda_1 \in T_y^*M, \qquad\nu \in \{0,1\},
\end{equation}
where $\circ$ denotes the composition of linear maps and $D$ the (Fr\'echet) differential. If $\gamma_u : [0,1] \to  M$ with control $u \in \mathcal{U}$ is an horizontal curve (not necessarily minimizing), we say that a non-zero pair $(\lambda_1,\nu) \in T_y^*M \times \{0,1\}$ is a \emph{Lagrange multiplier} for $\gamma_u$ if \eqref{eq:multipliers} is satisfied. The multiplier $(\lambda_1,\nu)$ and the associated curve $\gamma_u$ are called \emph{normal} if $\nu = 1$ and \emph{abnormal} if $\nu = 0$. Observe that Lagrange multipliers are not unique, and a horizontal curve may be both normal \emph{and} abnormal. Observe also that $\gamma_u$ is an abnormal curve if and only if $u$ is a critical point for $E_x$. In this case, $\gamma_u$ is also called a \emph{singular curve}. The following characterization is a consequence of the Lagrange multipliers rule, and can also be seen as a specification of the Pontryagin Maximum Principle to the sub-Riemannian length minimization problem.
\begin{theorem} \label{t:utile}
Let $\gamma_u :[0,1] \to M$ be an horizontal curve joining $x$ with $y$. A non-zero pair $(\lambda_1,\nu) \in T_y^*M \times \{0,1\}$ is a Lagrange multiplier for $\gamma_u$ if and only if there exists a Lipschitz curve $\lambda(t) \in T_{\gamma_u(t)}^*M$ with $\lambda(1) = \lambda_1,$ such that
\begin{itemize}
\item[(N)] if $\nu = 1$ then $\dot{\lambda}(t) = \vec{H}(\lambda(t))$,
\item[(A)] if $\nu =0 $ then $\sigma(\dot\lambda(t), T_{\lambda(t)} \distr^\perp) = 0$,
\end{itemize}
where $\distr^\perp \subset T^*M$ is the set of covectors that annihilate the distribution.
\end{theorem}
In the first (resp.\ second) case, $\lambda(t)$ is called a \emph{normal} (resp.\ \emph{abnormal}) \emph{extremal}. Normal extremals are   integral curves $\lambda(t)$ of $\vec{H}$. As such, they are smooth, and characterized by their \emph{initial covector} $\lambda = \lambda(0)$. A geodesic is normal (resp.\ abnormal) if admits a normal (resp.\ abnormal) extremal. On the other hand, it is well-known that the projection $\gamma_\lambda(t) = \pi(\lambda(t))$ of a normal extremal is locally minimizing, hence it is a normal geodesic. The \emph{exponential map} at $x \in M$ is the map $\exp_x : T_x^*M \to M$, which assigns to $\lambda \in T_x^*M$ the final point $\pi(\lambda(1))$ of the corresponding normal geodesic. The curve $\gamma_\lambda(t)=\exp_x(t \lambda)$, for $t \in [0,1]$, is the normal geodesic corresponding to $\lambda$, which has constant speed $\|\dot\gamma_\lambda(t)\| = \sqrt{2H(\lambda)}$ and length $\ell(\gamma|_{[t_1,t_2]}) = \sqrt{2H(\lambda)}(t_2-t_1)$.

\begin{definition}\label{d:ideal}
A sub-Riemannian structure $(\distr,g)$ on $M$ is \emph{ideal} if the metric space $(M,d_{SR})$ is complete and there exists no non-trivial abnormal length minimizers.\footnote{This means that the only possible abnormal length minimizers are constant curves.}
\end{definition}
The above terminology was introduced in \cite{RiffordCarnot,riffordbook}. By \cite[Sec.\ 5.6]{montgomerybook}, all complete fat structures are ideal. Moreover, the ideal assumption is generic, when the rank of the distribution is at least 3, in the following sense.

\begin{prop}[{\cite[Thm.\ 2.8]{CJT-goh}}]\label{p:generic} Let $k \geq 3$ be a positive integer, and $\mathcal{G}_{k}$ be the set of sub-Riemannian structures $(\distr, g)$ on $M$ with $\rank \distr=k$, endowed with the Whitney $C^{\infty}$ topology. There exists an open dense subset $W_{k}$ of $\mathcal{G}_{k}$ such that every element of $W_{k}$ does not admit non-trivial abnormal minimizers. 
\end{prop}
Next, we recall the definition of conjugate points.
\begin{definition}\label{def:conj}
Let $\gamma:[0,1]\to M$ be a normal geodesic with initial covector $\lambda \in T_x^*M$, that is $\gamma(t) = \exp_x(t\lambda)$. We say that $y=\exp_x(\bar{t}\lambda)$ is a \emph{conjugate point} to $x$ along $\gamma$ if $\bar{t}\lambda$ is a critical point for $\exp_x$.

Given a normal geodesic $\gamma:[0,1]\to M$ and $0\leq s< t\leq 1$, we say that $\gamma(s)$ and $\gamma(t)$ are \emph{conjugate} if $\gamma(t)$ is conjugate to $\gamma(s)$ along $\gamma|_{[s,t]}$.
\end{definition}

In the Riemannian setting, conjugate points along a geodesic are isolated, and geodesics cease to be minimizers after the first conjugate point. In the general sub-Riemannian setting, the picture is more complicated, but this result remains valid for geodesics that do not contain abnormal segments.

\begin{definition}\label{d:noabseg}
A normal geodesic $\gamma:[0,1]\to M$ \emph{contains no abnormal segments} if for every $0\leq s_{1}< s_{2}\leq 1$ the restriction $\gamma|_{[s_{1},s_{2}]}$ is not abnormal.
\end{definition}

\begin{theorem}[Conjugate points and minimality] \label{t:noconj}
Let $\gamma: [0,1] \to M$ be a minimizing geodesic which does not contain abnormal segments. Then $\gamma(s)$ is not conjugate to $\gamma(s')$ for every $s,s'\in [0,1]$ with $|s-s'|<1$.
\end{theorem}
Theorem \ref{t:noconj} is not new, and well-known to experts. We provide a self-contained proof in Appendix \ref{app:B}, following the arguments of \cite{nostrolibro} (see also \cite{S-index}). Notice that, as in the Riemannian case, $\gamma(1)$ can be conjugate to $\gamma(0)$ along $\gamma$.

\subsection{Regularity of the sub-Riemannian distance}

We recall now some basic regularity properties of the sub-Riemannian distance.
\begin{definition}
Let $(\distr,g)$ be a complete sub-Riemannian structure on $M$, and $x \in M$. We say that $y \in M$ is a \emph{smooth point} (with respect to $x$) if there exists a unique minimizing geodesic joining $x$ with $y$, which is not abnormal, and the two points are not conjugate along such a curve. The \emph{cut locus} $\cut(x)$ is the complement of the set of smooth points with respect to $x$. The \emph{global cut-locus} of $M$ is
\begin{equation}
\cut(M) := \{ (x,y) \in M \times M \mid y \in \cut(x)\}.
\end{equation}
\end{definition}
We have the following fundamental result \cite{Agrasmoothness,RT-MorseSard}.
\begin{theorem}\label{t:pointofsmoothness}
The set of smooth points is open and dense in $M$, and the squared sub-Riemannian distance is smooth on $M \times M \setminus \cut(M)$.
\end{theorem}

\section{Jacobi fields and second differential}\label{s:seconddiff}

If $x \in M$ is a critical point for $f \in C^\infty(M)$, one can define the Hessian of $f$ as
\begin{equation}
\Hess(f)|_x : T_{x}M\times T_{x}M\to \R,\qquad \Hess(f)|_x(v,w)=V(W(f))(x),
\end{equation}
where $V,W$ are local vector fields such that $V(x)=v$ and $W(x)=w$. Since $x$ is a critical point, the definition is well posed, and $\Hess(f)|_x$ is a symmetric bilinear map. The quadratic form associated with the second differential of $f$ at $x$ which, for simplicity, we denote by the same symbol $\Hess(f)|_x : T_x M \to \mathbb{R}$, is
\begin{equation}
\Hess(f)|_x(v)=\frac{d^{2}}{dt^{2}}\bigg|_{t=0} f(\gamma(t)),\qquad \gamma(0)=x,\quad \dot \gamma(0)=v.
\end{equation}
When $x \in M$ is not a critical point, we define the second differential of $f$ as the differential of $df$, thought as a smooth section of $T^*M$.

\begin{definition}[Second differential at non-critical points]\label{d:secdif}
Let $f \in C^\infty(M)$, and
\begin{equation}
df: M\to T^{*}M, \qquad df: x\mapsto d_{x}f.
\end{equation}
The \emph{second differential} of $f$ at $x \in M$ is the linear map
\begin{equation} \label{eq:2dxx}
d^{2}_{x}f:= d_{x}(df): T_{x}M\to T_{\lam}(T^{*}M),
\end{equation}
where $\lambda = d_xf \in T^*M$.
\end{definition}

The image of $df: M\to T^{*}M$ is a Lagrangian\footnote{A Lagrangian submanifold of $T^*M$ is a submanifold such that its tangent space is a Lagrangian subspace of the symplectic space $T_{\lambda}(T^*M)$. A subspace $L \subset \Sigma$ of a symplectic vector space $(\Sigma,\sigma)$ is Lagrangian if $\dim L = \dim\Sigma/2$ and $\sigma|_{L} = 0$.} submanifold of $T^*M$. Thus, the image of the second differential $d^{2}_{x}f (T_x M)$ at a point $x$ is the tangent space of $df(M)$ at $\lam=d_{x}f$, which is an $n$-dimensional Lagrangian subspace of $T_{\lam}(T^{*}M)$ transverse to  $T_\lam(T^*_x M)$. Letting $\pi: T^*M \to M$ be the cotangent bundle projection, and since $\pi\circ df=\mathrm{id}_{M}$, we have that $\pi_{*}\circ d^{2}_{x}f=\mathrm{id}_{T_{x}M}$.

\begin{lemma}\label{l:affine}
Let $\lam \in T_x^*M$. The set $\mathcal{L}_\lambda:= \{d_x^2 f \mid f \in C^{\infty}(M),\, d_x f = \lambda\}$ is an affine space over the vector space of quadratic forms on $T_{x}M$.
\end{lemma}
The above lemma follows from the fact that if $f_1, f_2 \in C^\infty(M)$ satisfy $d_x f_1 = d_x f_2 =\lam$, then $x$ is a critical point for $f_1 - f_2$ and one can define the difference between $d^2_x f_1$ and $d^2_x f_2$ as the quadratic form $\Hess(f_1 - f_2)|_x$.


\begin{remark}\label{r:secdifftwice}
Definition~\ref{d:secdif} can be extended to any $f :M \to \R$ twice differentiable at $x$. In this case, fix local coordinates around $x$, and let $b(x) \in \R^n$ and $A(x) \in \mathrm{Sym}(n\times n)$ such that
\begin{equation}
\lim_{t \downarrow 0} \frac{f(x+tv)-f(x) - t b(x) \cdot v - \tfrac{t^2}{2} v\cdot A(x)v}{t^2} =0, \qquad \forall v \in \R^n.
\end{equation}
Letting $(q,p) \in \R^{2n}$ denote canonical coordinates around $d_x f \in T^*M$, we define
\begin{equation}
d_x^2 f \left(\partial_{q_i}\right) : = \partial_{q_i}|_{d_x f} + \sum_{j=1}^n A_{ij}\partial_{p_j}|_{d_x f}, \qquad \forall i=1,\dots,n.
\end{equation}
This definition is well posed, i.e., it does not depend on the choice of coordinates.
\end{remark}

\subsection{Sub-Riemannian Jacobi fields}
Let $\lambda_t= e^{t\vec{H}}(\lambda_0)$, $t \in [0,1]$ be an integral curve of the Hamiltonian flow. For any smooth vector field $\xi(t)$ along $\lambda_t$, the dot denotes the Lie derivative in the direction of $\vec{H}$, namely
\begin{equation}
\dot{\xi}(t) := \left.\frac{d}{d\eps}\right|_{\eps=0} e^{-\eps \vec{H}}_* \xi(t+\eps).
\end{equation}
A vector field $\J(t)$ along $\lambda_t$ is a \emph{Jacobi field} if it satisfies the equation
\begin{equation}\label{eq:defJF}
\dot{\J} = 0.
\end{equation}
Jacobi fields along $\lambda_t$ are of the form $\J(t) = e^{t\vec{H}}_* \J(0)$, for some unique initial condition $\J(0) \in T_{\lambda_0} (T^*M)$, and the space of solutions of \eqref{eq:defJF} is a $2n$-dimensional vector space. On $T^*M$ we define the smooth sub-bundle with Lagrangian fibers:
\begin{equation}
\ver_{\lambda} := \ker \pi_*|_{\lambda} = T_\lambda(T^*_{\pi(\lambda)} M) \subset T_{\lambda}(T^*M), \qquad \lambda \in T^*M,
\end{equation}
which we call the \emph{vertical subspace}. In this formalism, letting
\begin{equation}
\gamma(t)=\exp_x(t \lambda_0) = \pi \circ e^{t\vec{H}}(\lambda_0), \qquad t \in [0,1],
\end{equation}
we have that $\gamma(s)$ is conjugate with $\gamma(0)$ along the normal geodesic $\gamma$ if and only if the Lagrangian subspace $e^{s\vec{H}}_*\ver_{\lambda_0} \subset T_{\lambda_s}(T^*M)$ intersects $\ver_{\lambda_s}$ non-trivially.

The next statement generalizes the well known Riemannian fact that, in absence of conjugate points, Jacobi fields are either determined by their value and the value of the covariant derivative in the direction of the given geodesic at the initial time, or by their value at the final and initial times.
\begin{lemma}\label{l:mixed}
Assume that, for $s \in (0,1]$, $\gamma(0)$ is not conjugate to $\gamma(s)$ along $\gamma$. Let $\mathcal{H}_{\lambda_i} \subset T_{\lambda_i}(T^*M)$ be transverse to $\ver_{\lambda_i}$, for $i=0,s$. Then for any pair $(J_0,J_s) \in \mathcal{H}_{\lambda_0} \times \mathcal{H}_{\lambda_s}$, there exists a unique Jacobi field $\mathcal{J}(t)$ along $\lambda_t$, $t \in [0,1]$, such that the projection of $\mathcal{J}(i)$ on $\mathcal{H}_{\lambda_i}$ is equal to $J_i$, for $i=0,s$.
\end{lemma}
\begin{proof}
The condition at $t=0$ implies that $\J(0) \in J_0 + \ver_{\lambda_0}$ (an affine space). By definition of Jacobi field, $\J(t) = e^{t\vec{H}}_* \J(0)$, in particular $\J(s) \in e^{s\vec{H}}_* J_0 + e^{s\vec{H}}_* \ver_{\lambda_0}$. By the non-conjugate assumption and since $T_{\lambda_s}(T^*M) = \ver_{\lambda_s} + \mathcal{H}_{\lambda_s}$, the projection of the affine space $e^{s\vec{H}}_* J_0 + e^{s\vec{H}}_* \ver_{\lambda_0}$ on $\mathcal{H}_{\lambda_s}$ is a bijection, yielding the statement.
\end{proof}

\subsection{Jacobi matrices}\label{s:jacmatrices}

We introduce a formalism to describe families of subspaces generated by Jacobi fields. Let $\gamma :[0,1] \to M$ be a normal geodesic, projection of $\lambda_t = e^{t\vec{H}}(\lambda_0)$, for some $\lambda_0 \in T^*M$. Consider the family of $n$-dimensional subspaces generated by a set of independent Jacobi fields $\J_1(t),\dots,\J_n(t)$ along $\lambda_t$, that is
\begin{equation}
\mathcal{L}_t=\spn\{\J_1(t),\dots,\J_n(t)\} \subset T_{\lambda_t}(T^*M).
\end{equation}
Since $\mathcal{L}_t = e^{t\vec{H}}_* \mathcal{L}_0$, then $\mathcal{L}_t$ is Lagrangian if and only if it is Lagrangian at time $t=0$.

Notice that $\mathcal{L}_t$ can be regarded as a smooth curve in a suitable (Lagrange) Grassmannian bundle over $T^*M$. We do not pursue this approach here, and we opt for an extrinsic formulation based on Darboux frames. To this purpose, and in order to exploit the symplectic structure of $T^*M$, fix a Darboux moving frame along $\lambda_t$, that is a collection of smooth vector fields $E_1(t),\dots,E_n(t),F_1(t),\dots,F_n(t)$ such that
\begin{equation}
\sigma(E_i,F_j) - \delta_{ij} = \sigma(E_i,E_j) = \sigma(F_i,F_j) = 0, \qquad \forall i,j=1,\ldots,n,
\end{equation}
and such that the $E_1(t),\ldots,E_n(t)$ generate the vertical subspace $\mathcal{V}_{\lambda_t}=\ker \pi_{*}|_{\lambda_t}$:
\begin{equation}
\mathcal{V}_{\lambda_t}= \spn\{E_1(t),\dots,E_n(t)\}.
\end{equation}
We also denote with $X_i(t):=\pi_* F_i(t)$ the corresponding moving frame along the geodesic $\gamma$. In this case, we say that $E_i(t),F_i(t)$ is a \emph{Darboux lift} of $X_i(t)$. Notice that any smooth moving frame along a normal geodesic admits a Darboux lift along a corresponding normal extremal.

We identify $\mathcal{L}_t = \spn\{\J_1(t),\dots,\J_n(t)\}$ with a smooth family of $2n \times n$ matrices
\begin{equation}\label{eq:jacmatrix}
\mathbf{J}(t) = \begin{pmatrix}
M(t) \\
N(t)
\end{pmatrix}, \qquad t \in [0,1],
\end{equation}
such that, with respect to the given Darboux frame, we have
\begin{equation}\label{eq:howtowriteJ}
\J_{i}(t) = \sum_{j=1}^n E_{j}(t)  M_{ji}(t) +  F_{j}(t)N_{ji}(t), \qquad \forall i=1,\dots,n.
\end{equation}
We call $\mathbf{J}(t)$ a \emph{Jacobi matrix}, while the $n\times n$ matrices $M(t)$ and $N(t)$ represent respectively its ``vertical'' and ``horizontal'' components with respect to the decomposition induced by the Darboux moving frame
\begin{equation}
T_{\lambda_t}(T^*M) = \mathcal{H}_{\lambda_t} \oplus \mathcal{V}_{\lambda_t}, \qquad \text{with} \qquad \mathcal{H}_{\lambda_t} := \spn\{F_1(t),\dots,F_n(t)\}.
\end{equation}

The following property is fundamental for the following.
\begin{lemma}\label{l:subspacesriccati}
There exist smooth families of matrices $A(t),B(t),R(t)$, $t \in[0,1]$, with $B(t),R(t)$ symmetric and $B(t) \geq 0$, such that for any Jacobi matrix $\mathbf{J}(t)$, we have
\begin{equation}\label{eq:mainequation}
\frac{d}{dt} \begin{pmatrix}
M \\
N
\end{pmatrix} = \begin{pmatrix}
-A(t) & - R(t) \\
B(t) & A(t)^*
\end{pmatrix} \begin{pmatrix}
M \\
N
\end{pmatrix}.
\end{equation}
On any interval $I \subseteq [0,1]$ such that $M(t)$ is non-degenerate, the matrix $W(t):= N(t)M(t)^{-1}$ satisfies the Riccati equation
\begin{equation}
\dot{W} = B(t) + A(t)^* W + WA(t) + WR(t) W.
\end{equation}
The associated family of subspaces $\mathcal{L}_t$ is Lagrangian if and only if $W(t)$ is symmetric.
\end{lemma}
\begin{proof}
By completeness of the frame, there exist smooth matrices $A(t),B(t),C(t)$  and $R(t)$ such that, for all $t \in [0,1]$, it holds
\begin{equation}\label{eq:mainequation-proof}
\dot{E}  = E \cdot A(t)  -  F\cdot B(t), \qquad \dot{F}  =  E \cdot R(t) -  F \cdot C(t)^*.
\end{equation}
The notation in \eqref{eq:mainequation-proof} means that $\dot{E}_i = \sum_{j=1}^n  E_j  A(t)_{ji}- F_j B(t)_{ji}$, and similarly for $\dot{F}_i$. For $n$-tuples $V$, $W$, the pairing $\sigma(V,W)$ denotes the matrix $\sigma(V_i,W_j)$. In this notation, $\sigma(V,W)^* = -\sigma(W,V)$. Thanks to the Darboux condition, we obtain
\begin{equation}
C(t) = \sigma_{\lambda_t}(\dot{F},E) = -\sigma_{\lambda_t}(F,\dot{E}) = A(t).
\end{equation}
The symmetry of $R(t)$ and $B(t)$ follows similarly. Moreover, we have
\begin{equation}\label{eq:Bpositive}
B(t) = \sigma_{\lambda_t}(\dot{E},E) = 2H (E,E) \geq 0.
\end{equation}
Here, $H$ is the Hamiltonian seen as a fiber-wise bilinear form on $T^*M$, and we identify $T_{\g_t}^* M \simeq T_{\lambda_t}(T^*_{\g_t}M)$. The second equality in \eqref{eq:Bpositive} follows from a direct computation in canonical coordinates on $T^*M$. Observe that $B(t)$ has a non-trivial kernel if and only if the structure is not Riemannian. Finally, equation \eqref{eq:mainequation} follows from \eqref{eq:mainequation-proof}, \eqref{eq:howtowriteJ} and the Jacobi equation $\dot{\J}_i(t) = 0$. The claim about Riccati equation is proved by direct verification. 

Using \eqref{eq:howtowriteJ}, the Jacobi fields $\J_1(t), \dots, \J_n(t)$ associated with the Jacobi matrix $\mathbf{J}(t)$ generate a family of Lagrangian subspaces if and only if
\begin{equation}
0 = \sigma_{\lambda_t}(\mathcal{J},\mathcal{J}) = M(t)^*N(t) - N(t)^* M(t).
\end{equation}
The above identity is equivalent to the symmetry of $W(t)$. 
\end{proof}

In Riemannian geometry, the forthcoming manipulations are greatly simplified thanks to the existence of a particular Darboux frame, such that $A(t) = \mathbbold{0}$, $B(t) = \mathbbold{1}$ and where $R(t)$ represents the Riemannian sectional curvature of all $2$-planes containing $\dot\gamma(t)$. In the sub-Riemannian case, such a convenient frame is not available in full generality. To circumvent this problem we ``lift'' the problem on the cotangent bundle and avoid to pick some particular frame. See \cite{BR-comparison,BR-connection} for details.

\subsection{Special Jacobi matrices} \label{s:sjm}
Fix a normal geodesic $\gamma :[0,1] \to M$ and a smooth moving frame $E_1(t),\dots,E_n(t), F_1(t),\dots,F_n(t)$ along the corresponding extremal. Any Jacobi matrix is uniquely defined by its value at some intermediate time $\mathbf{J}(s)$, for $s \in [0,1]$. The following special Jacobi matrices will play a prominent role:
\begin{align}
\mathbf{J}^\vers_s(t) & = \begin{pmatrix}
M^\vers_s(t) \\
N^\vers_s(t)
\end{pmatrix}, \quad \text{such that} \quad \mathbf{J}^\vers_s(s) = \begin{pmatrix}
\mathbbold{1} \\ \mathbbold{0}
\end{pmatrix}, & \text{(``vertical'' at time $s$)}, \\
\mathbf{J}^\hors_s(t) & = \begin{pmatrix}
M^\hors_s(t) \\
N^\hors_s(t)
\end{pmatrix}, \quad \text{such that} \quad \mathbf{J}^\hors_s(s) = \begin{pmatrix}
\mathbbold{0} \\ \mathbbold{1}
\end{pmatrix}, & \text{(``horizontal'' at time $s$)},
\end{align}
representing, respectively, the families of Lagrange subspaces 
\begin{align*}
e^{(t-s)\vec{H}}_* \spn\{E_1(s),\dots,E_n(s)\} \qquad \text{and} \qquad e^{(t-s)\vec{H}}_* \spn\{F_1(s),\dots,F_n(s)\}.
\end{align*}
\begin{rmk}\label{r:readingconjugate}
Let $s_1,s_2 \in [0,1]$. Then $\gamma(s_1)$ is conjugate to $\gamma(s_2)$ along $\gamma$ if and only if at least one (and then both) the matrices $N^\vers_{s_1}(s_2)$ and $N^\vers_{s_2}(s_1)$ are degenerate.
\end{rmk}
\section{Main Jacobian estimate} 

We now state our two main technical results, which will be proved together.

\begin{theorem}\label{t:FR}
Let $(\distr,g)$ be a sub-Rieman\-nian structure on $M$. Let $x \neq y \in M$ and assume that there exists a function $\phi: M \to \R$, twice differentiable at $x$, such that 
\begin{equation}\label{eq:FR-assump}
\frac{1}{2}d^2_{SR}(x,y) = -\phi(x), \qquad \text{and} \qquad \frac{1}{2}d^2_{SR}(z,y) \geq -\phi(z), \quad \forall z \in M.
\end{equation}
Assume moreover that the minimizing curve joining $x$ with $y$, which is unique and given by $\gamma(t)=\exp_x(t d_x \phi)$, does not contain abnormal segments. Then $x \notin \cut(y)$.
\end{theorem}

We will usually apply Theorem~\ref{t:FR} to situations in which $\phi$ is twice differentiable almost everywhere, in such a way that the map
\begin{equation}
T_t(z)= \exp_z(t d_z \phi), \qquad \mis-\mathrm{a.e.}\, z \in M,
\end{equation}
is well defined. The next result is an estimate for its Jacobian determinant at $x$.	
\begin{theorem}[Main Jacobian estimate]\label{p:mainjacest}
Under the same hypotheses of Theorem~\ref{t:FR}, let $\gamma(t) = \exp_x(t d_x \phi)$, with $t \in [0,1]$, be the unique minimizing curve joining $x$ with $y$, which does not contain abnormal segments. Then, the linear maps
\begin{equation}
d_x T_t : T_x M \to T_{\gamma(t)} M, \qquad d_x T_t := \pi_* \circ e^{t\vec{H}}_* \circ d^2_x \phi,
\end{equation}
satisfy the following estimate, for all fixed $s \in (0,1]$:
\begin{equation}\label{eq:mainjacest}
\det(d_xT_t)^{1/n} \geq \left(\frac{\det N_s^\vers(t)}{\det N_s^\vers(0)}\right)^{1/n} + \left(\frac{\det N_0^\vers(t)}{\det N_0^\vers(s)}\right)^{1/n} \det(d_x T_s)^{1/n}, \quad \forall \, t \in [0,s],
\end{equation}
where the determinants are computed with respect to some smooth moving frame along $\gamma$, and the matrices $N^\vers_s(t)$ are defined as in Section \ref{s:sjm}, with respect to some Darboux lift along the corresponding extremal.

Both terms in the right hand side of \eqref{eq:mainjacest} are non-negative for $t \in [0,s]$ and, for $t \in [0,s)$, the first one is positive. In particular $\det(d_x T_t) > 0$ for all $t \in [0,1)$.
\end{theorem}
\begin{rmk}
As a matter of fact, $\det(d_x T_1)$ can be zero. For example, fix $x \notin \cut(y)$. The assumptions of Theorem~\ref{p:mainjacest} are satisfied by any smooth function $\phi$ such that $\phi(z) =  - d^2_{SR}(z,y)/2$ for all $z$ in a neighbourhood $\mathcal{O}_x$ of $x$. In particular $\exp_z(d_z \phi) = \pi\circ e^{\vec{H}}(d_z \phi) = y$ for all $z \in \mathcal{O}_x$, and thus $d_x T_1 = 0$.
\end{rmk}

We first discuss the strategy of the proof of Theorems \ref{t:FR} and \ref{p:mainjacest}. 
It is well known that, if \eqref{eq:FR-assump} holds and $\phi$ is differentiable at $x$, there exists a unique minimizing curve joining $x$ with $y$, which is the normal geodesic $\gamma(t)=\exp_x(td_x \phi)$, $t \in [0,1]$, see e.g.\ \cite[Lemma 2.15]{riffordbook}. By Theorem~\ref{t:noconj}, there are no conjugate points along $\gamma$, except possibly the pair $\gamma(0)$ and $\gamma(1)$. Thanks to this observation, we first prove that \eqref{eq:mainjacest} holds for all $s<1$. Then, we prove that if $\gamma(1)$ is conjugate to $\gamma(0)$, the right hand side of \eqref{eq:mainjacest} tends to $+\infty$ for $s \uparrow 1$ and any fixed $t>0$,  hence $\det(d_xT_t)^{1/n}=+\infty$, leading to a contradiction. This implies that $\gamma(1)$ is not conjugate to $\gamma(0)$, yields the validity of \eqref{eq:mainjacest} for all $s \in (0,1]$, and proves that $y \notin \cut(x)$.

\begin{proof}[Proof of Theorems \ref{t:FR} and \ref{p:mainjacest}]
Let $\lambda(t):=e^{t\vec{H}}(d_x \phi)$, and $\gamma(t) = \pi(\lambda(t))$ the corresponding minimizing geodesic, with $t \in [0,1]$. Let $E_1(t),\dots,E_{n}(t),F_1(t),\dots,F_n(t)$ be a Darboux lift along $\lambda(t)$ of a smooth moving frame $X_1(t),\dots,X_n(t)$ along $\gamma(t)$, that is satisfying
\begin{equation}
\sigma(E_i,F_j) - \delta_{ij} = \sigma(E_i,E_j) = \sigma(F_i,F_j) = 0, \qquad \forall i,j=1,\ldots,n,
\end{equation}
with $X_i(t) = \pi_* F_i(t)$ and $\pi_* E_i(t) = 0$ for $i=1,\ldots,n$ and $t \in [0,1]$.

Since $\phi$ is twice differentiable at $x$, the family of Lagrangian subspaces $e^{t\vec{H}}_* \circ d_x^2 \phi(T_x M) \subset T_{\lambda(t)}(T^*M)$ is well defined for all $t \in [0,1]$, and is associated via the given Darboux frame to the Jacobi matrix
\begin{equation}\label{eq:subsp}
\mathbf{J}(t) = \begin{pmatrix} M(t) \\ N(t) \end{pmatrix}, \qquad \text{such that} \qquad e^{t\vec{H}}_* \circ d_x^2 \phi(X(0)) = E(t) \cdot M(t) + F(t) \cdot N(t).
\end{equation}
In particular, $d_xT_t (X(0)) = X(t) \cdot N(t)$. Let now $s \in (0,1)$, and consider the Jacobi matrices $\mathbf{J}^\vers_0$ and $\mathbf{J}^\vers_s$ of Section~\ref{s:sjm}. Since $\gamma(0)$ is not conjugate to $\gamma(s)$, we have
\begin{equation}\label{eq:eqforNN0}
e^{s\vec{H}}_* \ver_{\lambda(0)} \cap \ver_{\lambda(s)} = \{0\}, \qquad \forall s \in (0,1).
\end{equation}
Equivalently, $N_s^\vers(0)$ and $N_0^\vers(s)$ are invertible. By Lemma~\ref{l:mixed}, a Jacobi matrix is uniquely specified by its horizontal components $N(0)$ and $N(s)$, hence we have
\begin{equation}\label{eq:rewriting}
\mathbf{J}(t) = \mathbf{J}^\vers_s(t) N_s^\vers(0)^{-1}N(0) + \mathbf{J}^\vers_0(t)N_0^\vers(s)^{-1} N(s), \qquad t \in [0,1], \quad s \in (0,1).
\end{equation}
By construction $N(0)= \mathbbold{1}$, and the horizontal component of \eqref{eq:rewriting} reads
\begin{equation}\label{eq:eqforNN}
N(t) = N^\vers_s(t) N_s^\vers(0)^{-1} + N^\vers_0(t)N_0^\vers(s)^{-1} N(s), \qquad t \in [0,1], \quad s \in (0,1).\end{equation}
The next crucial lemma is a consequence of two facts: the non-negativity of the Hamiltonian, and assumption \eqref{eq:FR-assump}. We postpone its proof to Appendix \ref{a:cruciale}.
\begin{lemma}[Positivity]\label{l:keylemma1}
Under the assumptions of Theorem~\ref{t:FR}, there exists a smooth family of $n\times n$ matrices $K(t)=N_0^\vers(t)^{-1}$, defined for $t \in (0,1)$, such that, for all $s \in (0,1)$, we have
\begin{itemize}
\item[(a)] $\det K(t) >0$ for all $t \in (0,1)$,
\item[(b)] $K(t) N^\vers_s(t) N_s^\vers(0)^{-1}\geq 0$, for all $t \in (0,s]$,
\item[(c)] $K(t) N^\vers_0(t)N_0^\vers(s)^{-1} N(s) \geq 0$, for all $t \in (0,1)$.
\end{itemize}
Furthermore, if $\gamma(1)$ is not conjugate to $\gamma(0)$, the above properties hold for all $s \in (0,1]$ and $t \in (0,1]$.
\end{lemma}
Minkowski determinant theorem \cite[4.1.8]{Marcusinequalities} states that the function $A \mapsto (\det A)^{1/n}$ is concave on the set of $n\times n$ non-negative symmetric matrices. Thus, by multiplying from the left \eqref{eq:eqforNN} by the matrix $K(t)$ of Lemma~\ref{l:keylemma1}, we obtain
\begin{equation}\label{eq:N01}
\det(d_xT_t)^{1/n} \geq \left(\frac{\det N_s^\vers(t)}{\det N_s^\vers(0)}\right)^{1/n} + \left(\frac{\det N_0^\vers(t)}{\det N_0^\vers(s)}\right)^{1/n} \det(d_x T_s)^{1/n}, \quad t \in [0,s].
\end{equation}
 Notice that we do not use Lemma~\ref{l:keylemma1} to prove \eqref{eq:N01} for $t=0$, but in this case the inequality holds since $d_x T_0 = \mathrm{id}|_{T_xM}$ and $N^\vers_0(0) = \mathbbold{0}$. Hence, we obtain \eqref{eq:N01} for all $t \in [0,s]$, $s \in (0,1)$ and, if $\gamma(0)$ is not conjugate with $\gamma(1)$, also for $s=1$. We claim that, under the assumptions of Theorem~\ref{t:FR}, the latter is always the case.

By contradiction, assume that $\gamma(1)$ is conjugate to $\gamma(0)$. As we already remarked, $N_0^\vers(s)$ and $N_s^\vers(0)$ are non-degenerate for $s \in (0,1)$, but now $\det N_0^\vers(1) = \det N_{1}^\vers(0) = 0$. We claim that, for fixed $t \in (0,1)$, the right hand side of \eqref{eq:N01} tends to $+\infty$ for $s \uparrow 1$. To prove this claim, notice that both terms in the right hand side of \eqref{eq:N01} are non-negative thanks to Lemma~\ref{l:keylemma1}, and therefore
\begin{equation}\label{eq:N02}
\det(d_xT_t)^{1/n} \geq \left(\frac{\det N_s^\vers(t)}{\det N_s^\vers(0)}\right)^{1/n} \geq 0 , \qquad t \in [0,s].
\end{equation}
By Theorem~\ref{t:noconj}, $\gamma(t)$ is not conjugate to $\gamma(1)$ for any fixed $t \in (0,1)$. Hence $N_1^\vers(t)$ is not degenerate. On the other hand $\gamma(0)$ and $\gamma(1)$ are conjugate by our assumption, and $N^\vers_1(0)$ is degenerate. Taking the limit for $s \uparrow 1$, and since the left hand side of \eqref{eq:N02} does not depend on $s$, we obtain $\det(d_xT_t)^{1/n}= +\infty$, leading  to a contradiction.

We have so far proved that there is a unique minimizing geodesic joining $x$ with $y$, which is not abnormal, and $y$ is not conjugate to $x$. This means that $y \notin \cut(x)$, and concludes the proof of Theorem~\ref{t:FR}. Moreover,  \eqref{eq:eqforNN0}-\eqref{eq:eqforNN} hold for all $s \in (0,1]$ and $t \in [0,1]$. Therefore we can  apply Lemma~\ref{l:keylemma1} also for $s=1$, which completes the proof of \eqref{eq:mainjacest} for all $s \in (0,1]$ and $t \in [0,s]$.

We already proved that both terms in the right hand side of \eqref{eq:mainjacest} are non-negative for $t \in [0,s]$ (actually, the second one is non-negative for all $t \in [0,1]$, by part $(c)$ of Lemma~\ref{l:keylemma1}). Now that we also proved that $\gamma(t)$ is not conjugate to $\gamma(s)$ for all possible $0<|s-t|\leq 1$, we obtain that the first term cannot be zero except for $t=s$, and hence it is strictly positive for all $t \in [0,s)$.
\end{proof}

\subsection{Failure of semiconvexity at the cut locus: proof of Theorem~\ref{t:FRc-intro}} We say that a continuous function $f: M \to \R $ fails to be semiconvex at $x \in M$ if, in any set of local coordinates around $x$, we have
\begin{equation}\label{eq:theinfimum}
\inf_{0<|v|<1} \frac{f(x+v) + f(x-v) - 2f(x)}{|v|^2} = - \infty.
\end{equation}
Similarly, we say that $f$ fails to be semiconcave at $x \in M$ if 
\begin{equation}\label{eq:thesupremum}
\sup_{0<|v|<1} \frac{f(x+v) + f(x-v) - 2f(x)}{|v|^2} = + \infty.
\end{equation}
For background on locally semiconvex and semiconcave function we refer to \cite{cannarsabook}.

Let $\mathsf{d}^{2}_{y}:M\to \R$ be the sub-Riemannian squared distance from $y \in M$, that is
\begin{equation}
\mathsf{d}^{2}_{y}(z):=d^{2}_{SR}(z,y), \qquad \forall\,z \in M.
\end{equation}

The following result is a consequence of Theorem~\ref{t:FR}. For ideal structures, one can take $\Omega = M \setminus \{y\}$ in Corollary~\ref{t:FRc}, yielding Theorem~\ref{t:FRc-intro} presented in Section~\ref{s:intro}.

\begin{cor}\label{t:FRc}
Let $(\distr,g)$ be a sub-Riemannian structure on $M$. Let $y  \neq x$. Assume that there exists a neighbourhood $\Omega$ of $x$ such that all minimizing geodesics joining $y$ with points of $\Omega$ do not contain abnormal segments. Then $x \in \cut(y)$ if and only if the function $\mathsf{d}^{2}_{y}$ fails to be semiconvex at $x$.
\end{cor}

\begin{proof}
We prove the contrapositive of the above statement. First, if $x \notin \cut(y)$ then $f(z):=d^2_{SR}(z,y)$ is smooth in a neighbourhood of $x$ by Theorem~\ref{t:pointofsmoothness}, and hence the infimum in \eqref{eq:theinfimum} is finite. 

To prove the opposite implication, observe that that by \cite[Thm.\ 1, Thm.\ 5]{CR-Semiconcavity} $f$ is locally semiconcave in a neighbourhood of $x$ (for this property it suffices that no minimizing geodesic joining $y$ with points of $\Omega$ is abnormal). By standard properties of semiconcave functions (see \cite[Prop.\ 3.3.1]{CR-Semiconcavity}), there exist  local charts around $x$, and $p \in \R^{n}$, $C \in \R$ such that
\begin{equation}\label{eq:superdiff+}
f(x+v)-f(x) \leq p \cdot v + C|v|^2, \qquad \forall\, |v| < 1.
\end{equation}
Hence, assume that the infimum in \eqref{eq:theinfimum} is finite, that is there exists $K \in \R$ such that, in local charts around $x$, we have
\begin{equation}\label{eq:constantK}
f(x+v) + f(x-v) -2f(x) \geq  K |v|^2, \qquad \forall\, |v| < 1.
\end{equation}
Equations \eqref{eq:superdiff+}--\eqref{eq:constantK} yield that 
there exists a function $\phi: M \to \R$, twice differentiable at $x$, such that $f(z) \geq -\phi(z)$ for all $z \in M$, and $f(x) = -\phi(x)$. Our assumptions imply that the unique geodesic joining $x$ with $y$ does not contain abnormal segments. Hence, by Theorem~\ref{t:FR}, $y \notin \cut(x)$.
\end{proof}
\begin{rmk}\label{r:semisemi}
We observe the following general fact. For any sub-Riemannian structure, the infimum in \eqref{eq:theinfimum} for $f = \mathsf{d}^2_y$ is always finite for $x=y$. On the other hand (when the structure is not Riemannian) $\mathsf{d}^2_y$ fails to be semiconcave at $y$. The first statement follows trivially observing that $\mathsf{d}^2_y(y) = 0$ and $\mathsf{d}^2_y(z) \geq 0$. The second statement is a classical consequence of the Ball-Box theorem \cite{Jeanbook}.
\end{rmk}

\subsection{Regularity versus optimality: the non-ideal case}\label{s:regopt}

Theorem~\ref{t:FRc-intro} is false in the non-ideal case. In fact, consider the standard left-invariant sub-Riemannian structure on the product $\mathbb{H} \times \R$ of the three-dimensional Heisenberg group and the Euclidean line. Denoting points $x = (q,s) \in \mathbb{H} \times \R$, one has
\begin{equation}\label{eq:pitagorico}
d_{SR}^2((q,s),(q',s')) = d^2_{\mathbb{H}}(q,q') + |s-s'|^2.
\end{equation}

Without loss of generality, fix $(q',s') = (0,0)$. The set of points reached by abnormal minimizers is $\Abn(0) = \{(0,s) \mid s \in \R \}$. Here, the squared distance $\mathsf{d}^2_0(q,s):= d^2_{SR}((q,s),(0,0))$ is not smooth, but the infimum in \eqref{eq:theinfimum} is finite. In fact, the loss of smoothness is due to the failure of semiconcavity, see Remark~\ref{r:semisemi}.

Notice that abnormal geodesics joining the origin with points in $\Abn(0)$ are straight lines $t \mapsto (0,t)$, which are optimal for all times. Hence it seems likely that the failure of semiconvexity is related with the loss of optimality, while the failure of semiconcavity is related with the presence of abnormal minimizers. In the conclusion of this section, we formalize this latter statement.

\subsubsection{On the definition of cut locus} \label{s:openquestions}

In this paper, following \cite{FR-mass}, we defined the cut locus $\cut(x)$ as the set of points $y$ where the squared distance from $x$ is not smooth. Classically, the cut locus is  related with the loss of optimality of geodesics. To give a precise definition, we proceed as follows. First, we say that a geodesic $\gamma : [0,T] \to M$ (a horizontal curve which locally minimizes the energy between its endpoints) is maximal if it is not the restriction of a geodesic defined on a larger interval $[0,T']$. The cut time of a maximal geodesic is
\begin{equation}
t_{\mathrm{cut}}(\gamma) := \sup\{t > 0 \mid  \gamma|_{[0,t]} \text{ is a minimizing geodesic}\}.
\end{equation}
Assuming $(M,d_{SR})$ to be complete, we define the \emph{optimal cut locus} of $x \in M$ as
\begin{equation}
\cutopt(x) := \{ \gamma(t_{\mathrm{cut}}(\gamma)) \mid \gamma \text{ is a maximal geodesic starting at } x \}.
\end{equation}
In the ideal case, which includes the Riemannian case, it is well known that
\begin{equation}\label{eq:cutopt}
\cutopt(x) = \cut(x) \setminus \{x\}.
\end{equation}
For a general, complete sub-Riemannian structure $(\distr,g)$ on $M$, let $x \in M$ and define the following sets:
\begin{align*}
\SC^{-}(x) & := \{y \in M \mid \mathsf{d}^2_x \text{ fails to be semiconcave at $y$}\}, \\
\SC^{+}(x) & := \{y \in M \mid \mathsf{d}^2_x \text{ fails to be semiconvex at $y$}\},\\
\Abn(x) & := \{y \in M \mid \exists \text{ abnormal minimizing geodesic joining $x$ to $y$} \}.
\end{align*}
\textbf{Open questions.} Are the following equalities true?
\begin{align}
\cutopt(x) & = \SC^{+}(x),  \label{eq:op1}\\
\Abn(x) & = \SC^{-}(x). \label{eq:op2}
\end{align}
In the ideal case, \eqref{eq:cutopt} holds and $\Abn(x) = \{x\}$. Hence \eqref{eq:op1} follows from Corollary~\ref{t:FRc}, \eqref{eq:op2} follows from the results of \cite{CR-Semiconcavity} (where the general inclusion $\Abn(x)  \supseteq \SC^{-}(x)$ is proved). In particular, the following identities are true in the ideal case:
\begin{equation}
\cut(x) = \cutopt(x) \cup \Abn(x), \qquad \cut(x)=  \SC^{+}(x) \cup \SC^-(x). \label{eq:op3}
\end{equation}
We do not know whether \eqref{eq:op3} remain true in general. Notice that the first union in \eqref{eq:op3}, in general, is not disjoint \cite{RS-cut,BM-leftinvariant}.

We mention that \eqref{eq:op2} holds true for the Martinet flat structure (a rank-varying structure on $\R^3$). In fact, as proved in \cite{ABCK-martinet}, Martinet spheres possess outward corners in correspondence of points reached by abnormal minimizing geodesics, and this implies the loss of semiconcavity. The same characterization holds for the Engel group (a step $3$ and rank $2$ Carnot structure on $\R^4$), and for all free Carnot group of step $2$, as proved in \cite{MM-semiconc}. Finally, in \cite{MM-cut}, the authors proved the inclusion $\cutopt(x) \subseteq \SC^{+}(x)$ for the free Carnot group of step $2$ and rank $3$.		
\section{Optimal transport and interpolation inequalities}

The study of the Monge optimal transportation problem in sub-Riemannian geometry has been initiated in \cite{AR-Heisenberg,FJ-Heisenberg} for the Heisenberg group and subsequently developed in \cite{AL-transportation,FR-mass,Lee-JFA} for more general structures. 

\subsection{Sub-Riemannian optimal transport}\label{s:srtransport}

Let us fix a smooth (outer) measure $\mis$ on $M$, that is, defined by a positive tensor density. In particular $\mis$ is Borel regular and locally finite, hence Radon \cite{EG-Fine}.

The space of compactly supported probability measures on $M$ is denoted by $\mathcal{P}_c(M)$, while $\mathcal{P}_c^{ac}(M)$ is the subset of the absolutely continuous ones w.r.t.\ $\mis$. We denote by $\pi_i: M \times M \to M$, for $i=1,2$, the projection on the $i$-th factor. Furthermore, let $D= \{ (x,y) \in M \times M \mid x=y\}$.

Given two probability measures $\mu_0,\mu_1$ on $M$, we look for \emph{transport maps} between $\mu_0$ and $\mu_1$, that is measurable maps $T: M \to M$, such that $T_\sharp \mu_0 = \mu_1$. Furthermore, for a given cost function $c : M \times M \to \R$, we want to minimize the transport cost among all transport maps, that is solve the Monge problem:
\begin{equation}\label{eq:monge}
C_{\mathcal{M}}(\mu_0,\mu_1)=\min_{T_\sharp \mu_0 = \mu_1}\int_M c(x,T(x)) d\mis(x).
\end{equation}
Solutions of \eqref{eq:monge}, when they exist, are called \emph{optimal transport maps}.

We take from \cite[Thm.\ 3.2]{FR-mass} the main result about well-posedness of the Monge problem, and we specify it to the ideal setting. We give here a simplified but equivalent statement, which does not require a background in optimal transportation.
\begin{theorem}
[Well posedness of Monge problem]\label{t:well-posed}
Let $(\distr,g)$ be an ideal sub-Rieman\-nian structure on $M$, and $\mu_0 \in \mathcal{P}_c^{ac}(M), \,\mu_1 \in \mathcal{P}_c(M)$. Then there exists a unique transport map $T: M \to M$ such that $T_\sharp \mu_0 = \mu_1$, optimal w.r.t.\ the cost
\begin{equation}
c(x,y) = \frac{1}{2}d^2_{SR}(x,y),
\end{equation}
The map $T$ is characterized as follows. There exist a closed set $\mathcal{S}^\psi$ (the static set), an open set $\mathcal{M}^\psi = M \setminus \mathcal{S}^\psi$ (the moving set), and a function $\psi : M \to \R$ locally semiconvex in a neighbourhood of $\mathcal{M}^\psi \cap \supp(\mu_0)$, such that
\begin{itemize}
\item[(i)] For $\mu_0$-a.e.\ $x \in \mathcal{S}^\psi$ then $T(x) = x$;
\item[(ii)] For $\mu_0$-a.e.\ $x \in \mathcal{M}^\psi$ then $T(x) = y$ if and only if 
\begin{equation}\label{eq:csubdiff}
 \psi(x)+ c(x,y) \leq \psi(z) + c(z,y), \qquad \forall z \in M.
\end{equation}
\end{itemize}
In particular for $\mu_0$-a.e.\ $x \in M$ there exists a unique minimizing geodesic between $x$ and $T(x)$ given by
\begin{equation}
T_t(x) := \begin{cases}
\exp_x(t d_x \psi) & x \in \mathcal{M}^\psi \cap \supp(\mu_0), \\
x & x \in \mathcal{S}^\psi \cap \supp(\mu_0),
\end{cases} \qquad t \in [0,1].
\end{equation}
\end{theorem}
\begin{rmk}
In the language of optimal transport, $\psi$ is a Kantorovich potential associated with the problem and \eqref{eq:csubdiff} means that the pair $(x,y)$ belongs to the $c$-subdifferential $\partial_c\psi$.
\end{rmk}

\begin{rmk}
All the results proved in this section hold if, for given $\mu_0 \in \mathcal{P}^{ac}_c(M),\,\mu_1 \in \mathcal{P}_c(M)$, we replace the ideal hypothesis with the assumption that $(M,d_{SR})$ is complete and that there exists an open set $\Omega \subset M \times M$ such that $\supp(\mu_0 \times \mu_1) \subset \Omega$ and all minimizing geodesics with endpoints in $\Omega \setminus D$ contain no abnormal segments. This assumption is crucial for our Jacobian estimates, and it is used through Theorem~\ref{t:noconj}. 

It is currently unknown whether the Monge problem is well posed for general structures satisfying the so-called minimizing Sard property, i.e., where abnormal minimizers from any fixed point reach a set of measure zero. The reason is technical, since in this case one is not able to deduce enough regularity for the Kantorovich potentials associated with the optimal transport problem (see for example \cite{riffordbook}).
\end{rmk}

For what concerns the issue of regularity of $T$, in \cite[Thm.\ 3.7]{FR-mass}, Figalli and Rifford obtained a formula for the differential of the transport map akin the classical one of \cite{CEMS-interpolations} in terms of the Hessian of the distance, under additional hypothesis on the sub-Riemannian cut locus. More precisely, they assume that if $x \in \cut(y)$, then there exist at least two distinct minimizing geodesics joining $x$ with $y$. It turns out that the statement of \cite[Thm.\ 3.7]{FR-mass} holds with no assumptions on the cut locus in the ideal case. The differentiability result is most easily expressed in terms of approximate differential (see e.g. \cite[Sec.\ 5.5]{AGS-Gradientflows}).

\begin{definition}[Approximate differential]
We say that $f : M \to \R$ has an approximate differential at $x \in M$ if there exists a function $g :M \to \R$ differentiable at $x$ such that the set $\{f = g\}$ has density $1$ at $x$ with respect to $\mis$.\footnote{We compute the density using Euclidean balls in local charts around $x$. Since $\mis$ is smooth, it has positive density with respect to the Lebesgue measure in charts, hence the concept of density does not depend on the choice of charts. In particular, Lebesgue density theorem holds.} In this case, the approximate value of $f$ at $x$ is defined as $\tilde{f}(x) = g(x)$, and the approximate differential of $f$ at $x$ is defined as $\tilde d_x f := d_x g : T_x M \to T_{g(x)} M$.
\end{definition}

\begin{theorem}[Regularity of transport]\label{t:differentiability}
Under the same assumptions of Theorem~\ref{t:well-posed}, the map $T_t$ is differentiable $\mu_0$-a.e.\ on $\mathcal{M}^\psi\cap \supp(\mu_0)$, and it is approximately differentiable $\mu_0$-a.e. For $\mu_0$-a.e.\ $x \in M$ its approximate differential is
\begin{equation}\label{eq:differentiability}
\tilde{d}_x T_t = \begin{cases}
\pi_* \circ e^{t\vec{H}}_* \circ d^2_x \psi &  x \in \mathcal{M}^\psi \cap \supp(\mu_0), \\
\mathrm{id}|_{T_xM}  &  x \in \mathcal{S}^\psi \cap \supp(\mu_0).
\end{cases}
\end{equation}
Finally, $\det(\tilde{d}_x T_t) >0$ for all $t \in[0,1)$ and $\mu_0$-a.e.\ $x \in M$.
\end{theorem}
\begin{rmk}
If $\mathcal{S}^\psi$ is empty, which is the case for example when $\supp(\mu_0) \cap \supp(\mu_1)$ is empty, then $T_t$ is differentiable, and not only approximately differentiable, $\mu_0$-a.e.
\end{rmk}
\begin{proof}
The closed set $\mathcal{S}^\psi$ is measurable, $\mu_0 \ll \mis$, and $\mis$ is smooth. Then by applying Lebesgue density theorem we obtain that $T_t$ is approximately differentiable $\mu_0$-a.e.\ on $\mathcal{S}^\psi \cap \supp(\mu_0)$, with approximate differential given by the identity map. Furthermore, $\det(\tilde{d}_x T_t) = 1$ for $\mu_0$-a.e.\ $x \in \mathcal{S}^\psi \cap \supp(\mu_0)$.

We consider now the case $x \in \mathcal{M}^\psi$. Since local semiconvexity is invariant by diffeomorphisms, and since $\mis$ is smooth, Alexandrov theorem in $\R^n$ (see e.g.\ \cite[Thm.\ A.5]{FR-mass}) yields that $\psi$ is twice differentiable at $\mis$-a.e.\ point $x \in \mathcal{M}^\psi$, and its second differential can be computed according to Definition~\ref{d:secdif} and Remark~\ref{r:secdifftwice}. Hence, since $\mu_0 \ll \mis$, the map
\begin{equation}
x \mapsto T_t(x) = \exp_x(td_x \psi) = \pi \circ e^{t\vec{H}} \circ d_x \psi
\end{equation}
is differentiable for $\mu_0$-a.e.\ $x \in \mathcal{M}^\psi$, and its differential is computed as in \eqref{eq:differentiability}.

By Theorem~\ref{t:well-posed}, $y = T_1(x)$ if and only if  $\psi(z) + c(z,y) - \psi(x) - c(x,y) \geq 0$ for all $z \in M$.
In particular, one can apply Theorems~\ref{t:FR} and~\ref{p:mainjacest} with the function $\phi(z)  := \psi(z)- \psi(x) - c(x,T(x))$, at any point $x$ where $\psi$ is twice differentiable, i.e.\ $\mu_0$-almost everywhere on $\mathcal{M}^\psi$. In particular, for all such points $T(x) \notin \cut(x)$ and $\det(d_x T_t)>0$ for all $t \in [0,1)$.
\end{proof}
\begin{cor}
Under the same assumptions of Theorem~\ref{t:well-posed}, for $\mu_0$-a.e.\ $x \in \mathcal{M}^\psi$ we have $T(x) \notin \cut(x)$.
\end{cor}

As a consequence of Theorem~\ref{t:differentiability} and the estimate of Theorem~\ref{p:mainjacest}, we obtain an independent proof of \cite[Thm.\ 3.5]{FR-mass} about the absolute continuity of the Wasserstein geodesic between $\mu_0$ and $\mu_1$.

\begin{theorem}[Absolute continuity of Wasserstein geodesic]\label{t:abscont}
Under the same assumptions of Theorem~\ref{t:well-posed},  there exists a unique Wasserstein geodesic joining $\mu_0$ with $\mu_1$, given by $\mu_t = (T_t)_\sharp \mu_0$, for $t \in [0,1]$. Moreover, $\mu_t \in \mathcal{P}_c^{ac}(M)$ for all $t \in [0,1)$.
\end{theorem}
\begin{proof}
The existence and uniqueness part is standard and is done as in \cite[Sec.\ 6.3, first paragraph]{FR-mass}. Since $\mis$ is a smooth positive measure, the absolute continuity statement is equivalent to the absolute continuity of $\mu_t=(T_t)_\sharp \mu_0$ with respect to Lebesgue measure $\mathcal{L}^d$ in all local coordinate charts, where $\mu_0 = \rho \mathcal{L}^d$, for some $\rho \in L^1(\R^d)$. Thanks to Theorem~\ref{t:differentiability}, in charts $T_t : \R^d \to \R^d$ is approximately differentiable $\rho \mathcal{L}^d$-almost everywhere. Hence, the statement follows from the next lemma, which is a relaxed version of \cite[Lemma 5.5.3]{AGS-Gradientflows}. Its proof for completeness is in Appendix~\ref{a:ambrosio+}.

\begin{lemma}\label{l:ambrosio+}
Let $\rho \in L^1(\R^d)$ be a non-negative function. Let $f: \R^d \to \R^d$ be a measurable function and let $\Sigma_f$ be the set where it is approximately differentiable. Assume there exists a Borel set $\Sigma \subseteq \Sigma_f$ such that the difference $\{\rho >0 \} \setminus \Sigma$ is $\mathcal{L}^d$-negligible. Then $f_\sharp(\rho \mathcal{L}^d) \ll \mathcal{L}^d$ if and only if $|\det( \tilde{d}_x f) |>0$ for $\mathcal{L}^d$-a.e. $x \in \Sigma$. 

In this case, letting $f_\sharp (\rho \mathcal{L}^d)  = \rho_f \mathcal{L}^d$, we have
\begin{equation}
\rho_f(y) = \sum_{x \in \tilde{f}^{-1}(y) \cap \Sigma} \frac{\rho(x)}{|\det(\tilde{d}_x f)|}, \qquad y \in \R^n,
\end{equation}
with the convention that the r.h.s.\ is zero if $y \notin \tilde{f}(\Sigma)$. In particular, if we further assume that $\tilde{f}|_{\Sigma}$ is injective, then we have
\begin{equation}
\rho_f(\tilde{f}(x)) = \frac{\rho(x)}{|\det(\tilde{d}_x f)|}, \qquad  \forall x \in \Sigma.
\end{equation}

\end{lemma}
Notice that, in order to prove Theorem~\ref{t:abscont}, we need only the first implication of Lemma~\ref{l:ambrosio+}, that is if $|\det( \tilde{d}_x f)| >0$ for $\mathcal{L}^d$-a.e.\ $x \in \Sigma$, then $f_\sharp(\rho \mathcal{L}^d) \ll \mathcal{L}^d$.
\end{proof}

Thanks to Theorem~\ref{t:abscont}, for all $t \in [0,1)$, and also $t = 1$ if $\mu_1 \in \mathcal{P}^{ac}_c(M)$, let $\rho_t := d\mu_t / d\mis$. Then we have the following Jacobian identity. 

\begin{theorem}[Jacobian identity]
Under the same assumptions of Theorem~\ref{t:well-posed}, let $\rho_t = d\mu_t / d\mis$ for all $t \in [0,1)$, and also $t = 1$ if $\mu_1 \in \mathcal{P}^{ac}_c(M)$. Then, for $\mu_0$-a.e.\ $x \in M$, we have
\begin{equation} \label{eq:det0}
\frac{\rho_0(x)}{\rho_t(T_t(x))} = \begin{cases} \det (d_x T_t) \frac{\mis(X_1(t),\dots,X_n(t))}{\mis(X_1(0),\dots,X_n(0))} > 0 & x \in \mathcal{M}^\psi \cap \supp(\mu_0), \\
1 & x \in \mathcal{S}^\psi \cap \supp(\mu_0),
\end{cases}
\end{equation}
where $X_1(t),\dots,X_n(t)$ is some smooth moving frame along the geodesic $t \mapsto T_t(x)$, and the determinant of the linear map $d_x T_t : T_x M \to T_{T_t(x)}M$ is computed with respect to the given frame, that is
\begin{equation}
d_x T_t (X_i(0)) = \sum_{j=1}^n N_{ij}(t) X_j(t), \qquad \det(d_xT_t) := \det N(t).
\end{equation}
\end{theorem}
\begin{remark}In the Riemannian case,  when $\mis=\mis_{g}$ is the Riemannian volume, one can compute the determinant in \eqref{eq:det0} with respect to orthonormal frames, eliminating any dependence on the frame and obtaining the classical Monge-Amp\`ere equation.\end{remark}
\begin{proof}
By Theorem~\ref{t:abscont}, $\mu_t = (T_t)_\sharp \mu_0 \ll \mu_0$, hence one can repeat the arguments in the last paragraph of \cite[Sec.\ 6.4]{FR-mass}. Since $\mu_t \in \mathcal{P}_c^{ac}(M)$, there are optimal transport maps $T_t,S_t$ such that $(T_t)_\sharp \mu_0 = \mu_t$ and $(S_t)_\sharp \mu_t = \mu_0$. By uniqueness of the transport map, we obtain that $T_t$ is $\mu_0$-a.e.\ injective. Hence we can use the second part of Lemma~\ref{l:ambrosio+}, and in particular for $\mu_0$-a.e.\ $x \in \mathcal{M}^\psi$ we have
\begin{equation}
\frac{\rho_0(x)}{\rho_t(T_t(x))} = \det(d_x T_t) \frac{\mis(X_1(t),\dots,X_n(t))}{\mis(X_1(0),\dots,X_n(0))}.
\end{equation}
The extra term in the right hand side is due to the fact that we are not computing $d_x T_t$ in a set of local coordinates, but with respect to a smooth frame.
\end{proof}

\subsection{Distortion coefficients} 

Let $(\distr,g)$ be a sub-Riemannian structure on $M$, not necessarily ideal. The next lemma provides a general bound for the distortion coefficient (see definitions~\ref{d:introdist0} and \ref{d:introdist}).

\begin{lemma}[On-diagonal distortion bound]\label{l:diagonalcoefficient}
Let $\mis$ be a smooth measure on $M$. Then, for any $x \in M$, there exists $Q(x) \geq \dim(M)$ such that
\begin{equation}
\beta_t(x,x) \leq t^{Q(x)}, \qquad \forall t \in [0,1].
\end{equation}
\end{lemma}
\begin{proof}
The proof is based on privileged coordinates and dilations in sub-Riemannian geometry, see \cite{Jeanbook} for reference. Fix $x \in M$, and let $z$ denote a system of privileged coordinates on a neighbourhood $\mathcal{O}$ of $x$ (which we identify from now on with a relatively compact open set of $\R^n$, where $x$ corresponds to the origin). We claim that there exists $Q(x) \geq \dim(M)$ and a constant $C(x)>0$ such that, for sufficiently small $\varepsilon$, we have
\begin{equation}
\mis(\mathcal{B}_{\varepsilon}(x)) = \varepsilon^{Q(x)}C(x)\left(1+ O(\varepsilon)\right).
\end{equation}
This claim, together with the observation that $Z_t(x,\mathcal{B}_r(x)) \subseteq \mathcal{B}_{tr}(x)$, implies the statement. In order to prove the claim, in the given set of privileged coordinates, let $\mis = \mis(z) d\mathcal{L}(z)$ for some smooth, strictly positive function $\mis$. Assume $\varepsilon$ to be sufficiently small such that $\mathcal{B}_{\varepsilon}(x) \subset \mathcal{O}$. Let $\delta_{\varepsilon}$ be the non-homogeneous dilation defined by the given system of privileged coordinates at $x$, with non-holonomic weights $w_i(x)$, for $i=1,\ldots,n$, that is $\delta_\varepsilon(z_1,\dots,z_n) = (\varepsilon^{w_1} z_1,\dots,\varepsilon^{w_n} z_n)$. Notice that, in coordinates, $\det(d_z \delta_\varepsilon) = \varepsilon^{Q(x)}$, where $Q(x) = \sum_{i=1}^n w_i(x)$. 

As a consequence of the distance estimates in \cite[Thm.\ 2.2]{Jeanbook}, for all sufficiently small $\varepsilon >0$ there exists $\alpha(\varepsilon) \downarrow 0$ such that
\begin{equation}\label{eq:estimateballs}
\widehat{\mathcal{B}}_{(1-\alpha(\varepsilon))\varepsilon} \subseteq \mathcal{B}_{\varepsilon}(x) \subseteq \widehat{\mathcal{B}}_{(1+\alpha(\varepsilon))\varepsilon},
\end{equation}
where $\widehat{\mathcal{B}}$ denotes the ball of the nilpotent structure, centered at the origin, in this set of privileged coordinates. By the homogeneity with respect to $\delta_\varepsilon$, we have
\begin{equation}
\widehat{\mathcal{B}}_{1- \alpha(\varepsilon)} \subseteq \delta_{1/\varepsilon}(\mathcal{B}_{\varepsilon}(x)) \subseteq \widehat{\mathcal{B}}_{1+\alpha(\varepsilon)}.
\end{equation}
The above relation, and the monotonicity of the Lebesgue measure as a function of the domain, imply that there exists a constant $B(x) = \mathcal{L}^{n}(\widehat{\mathcal{B}_1})>0$, such that
\begin{equation}
\lim_{\varepsilon \to 0} \mathcal{L}^n(\delta_{1/\varepsilon}(\mathcal{B}_\varepsilon)(x)) = B(x).
\end{equation}
Hence, since $\delta_\varepsilon$ and $\mis$ are smooth, we have
\begin{align}
\mis(\mathcal{B}_\varepsilon) & = \int_{\mathcal{B}_\varepsilon} \mis(z) \mathcal{L}^n(dz) \\
& = \int_{\delta_{1/\varepsilon}(\mathcal{B}_{\varepsilon})} \mis(\delta_{\varepsilon}(z)) \det(d_z \delta_\varepsilon)\mathcal{L}^n(dz) \\
& = \varepsilon^{Q(x)} \mis(0) \int_{\delta_{1/\varepsilon}(\mathcal{B}_{\varepsilon})}\left(1+\varepsilon R(\delta_\varepsilon(z))\right) \mathcal{L}^n(dz) \\
& = \varepsilon^{Q(x)} \mis(0) B(x) \left(1 + O(\varepsilon)\right),
\end{align}
where $R(z)$ is a smooth remainder, and the remainder term $O(\varepsilon)$ possibly depends on $x$. This concludes the proof of the claim.
\end{proof}

\begin{lemma}[Computation of the distortion coefficients]\label{l:distortioncomputed}
Let $x,y \in M$, with $x \notin \cut(y)$. Let $X_1,\dots,X_n$ be a smooth frame along the unique geodesic from $x$ to $y$. Then, in terms of the Jacobi matrices defined in Section~\ref{s:sjm}, we have
\begin{align}
\beta_t(x,y)   & = \frac{\det N^\vers_0(t) }{\det N^\vers_0(1)} \frac{\mis(X_1(t),\dots,X_n(t))}{\mis(X_1(1),\dots,X_n(1))}, \qquad \forall\, t \in [0,1], \label{eq:primabeta}\\
\beta_{1-t}(y,x)   & = \frac{\det N^\vers_1(t) }{\det N^\vers_1(0)} \frac{\mis(X_1(t),\dots,X_n(t))}{\mis(X_1(0),\dots,X_n(0))}, \qquad \forall\,t \in [0,1]. \label{eq:secondabeta}
\end{align}
Moreover, $\beta_t(x,y)>0$, for all $t \in (0,1]$.
\end{lemma}
\begin{proof}[Proof of Lemma~\ref{l:distortioncomputed}]
We prove first \eqref{eq:primabeta}. For $t=0$ both sides are zero, hence let $t \in (0,1]$. Let $\lambda_0$ be the initial covector of the unique minimizing geodesic such that $\exp_x(\lambda_0) = y$. Since $x \notin \cut(y)$, there exists an open neighbourhood $\mathcal{O}$ of $y$ and $O \subset T_x^*M$ such that $\exp_x: O \to \mathcal{O}$ is a smooth diffeomorphism, and for all $\lambda' \in O$, the geodesic $t \mapsto \exp_x(t \lambda')$ is the unique minimizing geodesic joining $x$ with $y' = \exp_x(\lambda')$, and $y'$ is not conjugate with $x$ along such a geodesic. Assuming $r$ sufficiently small such that $\mathcal{B}_r(y) \subset \mathcal{O}$, let $A_r \subset O$ be the relatively compact set such that $\exp_x(A_r) = \mathcal{B}_r(y)$. The map $\exp^t_x(\cdot) = \exp_x(t\cdot)$ is a smooth diffeomorphism from $A_r$ onto $Z_t(x,\mathcal{B}_r(y))$. In particular, we have
\begin{equation}\label{eq:differentiation}
\beta_t(x,y)  = \lim_{r \downarrow 0} \frac{\int_{A_r} \exp_x^{t*} \mis}{\int_{A_r} \exp_x^{1*} \mis} = \frac{(\exp_x^{t*} \mis)(\lambda_0)}{(\exp_x^{1*} \mis)(\lambda_0)}.
\end{equation}
The right hand side of \eqref{eq:differentiation} is the ratio of two smooth tensor densities computed at $\lambda_0$. To compute it, we evaluate both factors on a $n$-tuple of independent vectors of $T_x^*M$. Thus, pick a Darboux frame $E_1(t),\dots,E_n(t),F_1(t),\dots,F_n(t) \in T_{\lambda(t)}(T^*M)$ such that $\pi_* E_i(t) = 0$ and $\pi_* F_i(t) = X_i(t)$ for all $t \in [0,1]$, $i=1,\dots,n$. Then,
\begin{equation}
(\exp_x^{t*} \mis)(E_1(0),\dots,E_n(0)) = \mis(\pi_* \circ e^{t\vec{H}}_* E_1(0),\dots,\pi_* \circ e^{t\vec{H}}_* E_1(0)).
\end{equation}
The $n$-tuple $\J_i(t)=e^{t\vec{H}}_* E_i(0)$, $i=1,\dots,n$, corresponds to the Jacobi matrix $\mathbf{J}^\vers_0(t)$, such that $M^\vers_0(0) = \mathbbold{1}$ and $N^\vers_0(0) = \mathbbold{0}$, defined in Section~\ref{s:sjm}. Thus,
\begin{equation}
(\exp_x^{t*} \mis)(E_1(0),\dots,E_n(0)) = \det N^\vers_0(t)\, \mis(X_1(t),\dots,X_n(t)).
\end{equation}
By replacing the above formula in \eqref{eq:differentiation}, we obtain $\beta_t(x,y)$. Since $\gamma(t)$ is not conjugate to $\gamma(0)$ for all $t \in (0,1]$, we have $\beta_t(x,y)>0$ on that interval. 

Formula~\eqref{eq:secondabeta} for $\beta_{1-t}(y,x)$ is deduced in a similar way and with some additional care, following the geodesic backwards starting from the final point. We sketch the proof for this case. Let $\gamma: [0,1] \to M$ be the unique minimizing geodesic from $x$ to $y$, with extremal $\lambda :[0,1] \to T^*M$. Of course, the unique minimizing geodesic from $y$ to $x$ is $\tilde{\gamma}(t) = \gamma(1-t)$. The corresponding normal extremal is $\tilde{\lambda}(t) = - \lambda(1-t)$. Consider the inversion map $\iota: T^*M \to T^*M$, such that $\iota(\lambda) = -\lambda$.  In particular if $E_i(t),F_i(t)$ are a Darboux frame along $\lambda(t)$, then $\tilde{E}_i(t):= -\iota_* E_i(1-t)$ and $\tilde{F}_i(t):= -\iota_* F_i(1-t)$ are a Darboux frame along $\tilde{\lambda}(t)$. Hence, we have
\begin{equation}
\beta_{1-t}(y,x) = \frac{\mis(\pi_* \circ e^{(1-t)\vec{H}}_* \tilde{E}_1(0),\dots,\pi_* \circ e^{(1-t)\vec{H}}_* \tilde{E}_n(0))}{\mis(\pi_* \circ e^{\vec{H}}_* \tilde{E}_1(0),\dots,\pi_* \circ e^{\vec{H}}_* \tilde{E}_n(0))}, \qquad \forall t \in [0,1].
\end{equation}
We conclude the proof by observing that the $n$-tuple
\begin{equation}
\tilde{\J}_i(t) = e^{(1-t)\vec{H}}_* \tilde{E}_i(0) = e^{(t-1)\vec{H}}_* E_i(1) , \qquad i=1,\dots,n,
\end{equation}
corresponds to the Jacobi matrix $\mathbf{J}_1^\vers(t) = \left(\begin{smallmatrix} M^\vers_1(t) \\ N^\vers_1(t)\end{smallmatrix}\right)$ in terms of $E_1(t),\dots,F_n(t)$.
\end{proof}

\subsection{Interpolation inequalities: proof of Theorem~\ref{t:jacointerpineq}} \label{s:iipt4}
For $\mu_0$-a.e.\ $x \in \mathcal{S}^\psi$, by Theorems~\ref{t:well-posed}-\ref{t:differentiability} we have $T_t(x) = x$ and $\rho_t(x) = \rho_0(x)$ for all $t \in [0,1]$. In this case the inequality follows from Lemma~\ref{l:diagonalcoefficient},  which implies that
\begin{equation}
\beta_{1-t}(x,x)^{1/n} + \beta_t(x,x)^{1/n} \leq (1-t)^{Q(x)/n} + t^{Q(x)/n} \leq 1, \qquad \forall t \in [0,1].
\end{equation}

Fix now $x \in \mathcal{M}^\psi$, such that (i) $\psi : M \to \R$ is twice differentiable, (ii) the Jacobian identity of Theorem~\ref{t:differentiability} holds. By the absolute continuity of $\mu_0$ w.r.t.\ $\mis$, properties (i)-(ii) are satisfied $\mu_0$-a.e.\ in $\mathcal{M}^\psi$. Letting $X_i(t)$ be a moving frame along the geodesic $T_t(x) = \exp_x(td_x\psi(x))$, we have
\begin{equation}\label{eq:mongamperer}
\frac{\rho_0(x)}{\rho_t(T_t(x))} = \det(d_xT_t)\frac{\mis(X_1(t),\ldots,X_n(t))}{\mis(X_1(0),\ldots,X_n(0))} > 0.
\end{equation}
Recall that, according to Theorem~\ref{t:well-posed}, $y = T_1(x)$ if and only if for all $z \in M$ one has $\psi(z) + c(z,y) - \psi(x) - c(x,y) \geq 0$.	One can apply Theorem~\ref{t:FR} with $\phi(z)  := \psi(z)- \psi(x) - c(x,T(x))$ at the point $x$, where $\psi$ is twice differentiable. Hence, Theorem~\ref{p:mainjacest} yields an estimate for the determinant of the linear map
\begin{equation}
d_x T_t = \pi_*\circ e^{t\vec{H}}_* \circ d^2_x \phi = \pi_*\circ e^{t\vec{H}}_* \circ d^2_x \psi : T_x M \to T_{T_t(x)}M,
\end{equation}
given by \eqref{eq:mainjacest} for $s=1$. Since $T(x) \notin \cut(x)$ we can use the expressions for the distortion coefficients $\beta_t(x,T(x))$ of Lemma~\ref{l:distortioncomputed}, and we obtain
\begin{equation}
\frac{\rho_0(x)^{1/n}}{\rho_t(T_t(x))^{1/n}} \geq \beta_{1-t}(y,x)^{1/n} + \beta_{t}(x,y)^{1/n} \left(\det(d_x T_1) \frac{\mis(X_1(1),\dots,X_n(1))}{\mis(X_1(0),\dots,X_n(0))}\right)^{1/n}.
\end{equation}
If $\mu_1 \in \mathcal{P}_c^{ac}(M)$ we can again use \eqref{eq:mongamperer} for $t=1$ to replace the term containing $\det(d_x T_1)$, thus proving Theorem~\ref{t:jacointerpineq} for $t \in [0,1]$. If $\mu_1 \in \mathcal{P}_c(M) \setminus \mathcal{P}_c^{ac}(M)$ then \eqref{eq:mongamperer} holds only for $t \in [0,1)$. In this case, in the first step, we simply omit the second term in \eqref{eq:mainjacest} (which is non-negative). \hfill \qed
\section{Geometric and functional inequalities}\label{s:inequalities}

The first consequence of the interpolation inequalities proved so far is a sub-Riemannian Borell-Brascamp-Lieb inequality, that is Theorem~\ref{t:srbbl}. Its proof follows, without any modification, as in \cite[Sec.\ 6]{CEMS-interpolations}, and makes use of the assumption \eqref{eq:asscems} for triple of points $(x,y,z)$ satisfying $y=T_{1}(x)$ and $z=T_{t}(x)$, for some transport map $T$. This justifies removing $\cut(M)$ from $A\times B$.

\begin{theorem}[Sub-Riemannian Borell-Brascamp-Lieb inequality] \label{t:srbbl}
Let $(\distr,g)$ be an ideal sub-Riemannian structure on a $n$-dimensional manifold $M$, equipped with a smooth measure $\mis$. Fix $t\in [0,1]$. Let $f,g,h:M\to \R$ be non-negative and $A,B\subset M$ Borel subsets such that $\int_{A}f \,d\mis=\int_{B}g \,d\mis=1$.  Assume that for every $(x,y)\in (A\times B)\setminus \cut(M)$  and $z\in Z_{t}(x,y)$,
\begin{equation}\label{eq:asscems}
\frac{1}{h(z)^{1/n}}\leq \left(\frac{\beta_{1-t}(y,x)}{f(x)}\right)^{1/n}+  \left(\frac{\beta_{t}(x,y)}{g(y)}\right)^{1/n}.
\end{equation}
Then $\int_{M} h\, d\mis \geq 1$.
\end{theorem}

Let $t \in [0,1]$ and $a,b \geq 0$. We introduce the $p$-mean for $p \in \R \cup \{\pm \infty\}$. When $p \neq 0, \pm \infty$ then
\begin{equation}\label{eq:pmean}
\mathcal{M}_{t}^{p}(a,b):=
\begin{cases}
\left( (1-t)a^{p}+t b^{p} \right)^{1/p}& \text{if } ab\neq 0,\\
0& \text{if } ab= 0.
\end{cases}
\end{equation}
The limit cases are defined as follows
\begin{equation}
 \mathcal{M}_{t}^{0}(a,b):=a^{1-t}b^{t},
\quad \mathcal{M}_{t}^{+\infty}(a,b):=\max\{a,b\}, \quad \mathcal{M}_{t}^{-\infty}(a,b):=\min\{a,b\}.
\end{equation}
The next result follows in a standard way from Theorem~\ref{t:srbbl}, by elementary properties of $\mathcal{M}_t^p$. Theorem~\ref{t:srbbl} can be recovered from Theorem~\ref{t:srpmean} by setting $p=-1/n$. The case $p=0$ is the so-called Pr\'ekopa-Leindler inequality.
\begin{theorem}[Sub-Riemannian $p$-mean inequality] \label{t:srpmean}
Let $(\distr,g)$ be an ideal sub-Riemannian structure on a $n$-dimensional manifold $M$, equipped with a smooth measure $\mis$. Fix $p\geq -1/n$ and $t\in [0,1]$. Let $f,g,h:M\to \R$ be non-negative and $A,B\subset M$ be Borel subsets such that $\int_{A}f \,d\mis=\|f\|_{L^{1}(M)}$ and  $\int_{B}g \,d\mis=\|g\|_{L^{1}(M)}$.  Assume that for every $(x,y)\in (A\times B)\setminus \cut(M)$  and $z\in Z_{t}(x,y)$,
\begin{equation}
h(z)\geq \mathcal{M}_{t}^{p}\left(\frac{(1-t)^{n}f(x)}{\beta_{1-t}(y,x)},\frac{t^{n}g(y)}{\beta_{t}(x,y)}\right). 
\end{equation}
Then, 
\begin{equation}
\int_{M} h\, d\mis \geq \mathcal{M}_{t}^{p/(1+np)}\left(\int_{M} f\, d\mis, \int_{M} g\, d\mis\right),
\end{equation}
with the convention that if $p=+\infty$ then $p/(1+np) = 1/n$, and if $p=-1/n$ then $p/(1+np) = -\infty$.
\end{theorem}

\subsection{Brunn-Minkowski inequality: proof of Theorem~\ref{t:bmsr-intro}}

For $t=0$ or $t=1$, inequality \eqref{eq:bmsr-intro} is trivially verified. Hence let $t \in (0,1)$. Assume first that $Z_{t}(A,B)$ is measurable, and set 
\begin{equation}\label{eq:scelta}
f=\frac{\beta_{1-t}(B,A)}{(1-t)^{n}}\chi_{A},\qquad 
g=\frac{\beta_{t}(A,B)}{t^{n}}\chi_{B},\qquad h=\chi_{Z_{t}(A,B)},
\end{equation}
where $\chi_S$ is the characteristic function of a set $S \subset M$ and $\beta_{t}(A,B)$ is defined in \eqref{eq:betatAB}. The assumption in Theorem \ref{t:srpmean} is satisfied with $p=+\infty$ since
for every $(x,y)\in (A\times B)\setminus \cut(M)$  and $z\in Z_{t}(x,y)$,
\begin{equation}
1=h(z)\geq \max \left\{\frac{(1-t)^{n}f(x)}{\beta_{1-t}(y,x)},\frac{t^{n}g(y)}{\beta_{t}(x,y)}\right\}=\max \left\{\frac{\beta_{1-t}(B,A)}{\beta_{1-t}(y,x)},\frac{\beta_{t}(A,B)}{\beta_{t}(x,y)}\right\}. 
\end{equation}
Then, we have (when $p=+\infty$ it is understood that $p/(1+np)=1/n$)
\begin{align*}
\mis(Z_{t}(A,B))=\int_{M} h\, d\mis &\geq \mathcal{M}_{t}^{1/n}\left(\int_{M} f\, d\mis, \int_{M} g\, d\mis\right)\\
&= \left(\beta_{1-t}(B,A)^{1/n} \mis(A)^{1/n}+ \beta_{t}(A,B)^{1/n} \mis(B)^{1/n}\right)^{n},
\end{align*}
which proves the required inequality.

Assume now that $Z_{t}(A,B)$ is not measurable.  
Since $\mis$ is Borel regular, there exists a measurable set $C$ such that $Z_{t}(A,B)\subset C$, with $\mis(Z_{t}(A,B))=\mis(C)$. We have clearly that $\chi_{C}\geq \chi_{Z_{t}(A,B)}$ and $\chi_{C}$ is measurable. The conclusion follows repeating the argument above replacing  $h_{Z_t(A,B)}$ with $h_C$ in \eqref{eq:scelta}.\hfill \qed

\subsection{Equivalence of inequalities: proof of Theorem \ref{t:equivalenza}}
Let $(\distr,g)$ be an ideal sub-Rieman\-nian structure on a $n$-dimensional manifold $M$, equipped with a smooth measure $\mis$, and $N \geq 1$. We prove that (i) $\Rightarrow$ (ii) $\Rightarrow$ (iii) $\Rightarrow$ (i). First, by plugging (i) in the result of Theorem~\ref{t:bmsr-intro} we obtain (ii). Furthermore, (iii) is a particular case of (ii) by considering only sets of the form $A=\{x\}$. Finally, (iii) implies (i) by choosing in the former $B = \mathcal{B}_r(y)$, and recalling Definition \ref{d:introdist} of $\beta_t(x,y)$. \hfill \qed



\section{Examples}\label{s:esempi}

In this section we discuss the distortion coefficients for some examples. The first one is the Heisenberg group.
 
\subsection{Heisenberg group}\label{s:heis}

The Heisenberg group $\mathbb{H}_{3}$ is the sub-Riemannian structure on $M= \R^3$ defined by the global generating vector fields
\begin{equation}
X_1 = \partial_x-\frac{y}{2}\partial_{z}, \qquad X_2 =  \partial_y+\frac{x}{2}\partial_{z}.
\end{equation}
The distribution has constant rank equal to two, and the sub-Riemannian structure is left-invariant with respect to the group product
\begin{equation}
(x,y,z)\star(x',y',z')=\left(x+x',y+y',z+z'+\frac12(xy'-y x')\right).
\end{equation}
The Heisenberg group is hence a Lie group and we equip it with the Lebesgue measure $\mis = \mathcal{L}^3$, which is a Haar measure. Thanks to the left-invariance of the sub-Riemannian structure, it is enough to compute the distortion coefficients when one of the two points is the origin.

In dual coordinates $(u,v,w,x,y,z)$ on $T^* \R^3$, the corresponding Hamiltonian is 
\begin{equation}
H(u,v,w,x,y,z) = \frac{1}{2}\left(\left(u - \frac{y}{2} w\right)^2 + \left(v+\frac{x}{2} w\right)^2\right).
\end{equation}
Hamilton equations can be explicitly integrated. In particular for an initial covector $\lambda_0 =u_0 dx + v_0 dy+w_{0}dz  \in T_{0}^*\R^3$, the exponential map from the origin $\exp^t_0(\lambda_{0}):= \exp_0(t\lambda_{0})$ reads $\exp_{0}^t(u_0,v_0,w_{0}) = (x(t),y(t),z(t))$, where
\begin{align}
x(t) & = \frac{u_0 \sin(w_0 t)+  v_{0} (\cos(w_0 t) -1)}{w_{0}},  \\
y(t) & = \frac{u_0(1- \cos(w_0 t)) + v_{0} \sin(w_0 t)}{w_{0}}, \\
z(t) & = (u_{0}^{2}+v_{0}^{2})\frac{w_{0}t-\sin (w_{0}t )}{2w_{0}^2}.
\end{align}
\begin{remark}
The geodesic flow is an analytic function of the initial data, and if $w_0=0$ the above equations are understood by taking the limit $w_0 \to 0$. We always adopt this convention in this section.
\end{remark}
In order to use Lemma~\ref{l:distortioncomputed} for the computation of the distortion coefficient, we choose the global Darboux frame induced by the global sections of $T(T^*\R^3)$:
\begin{equation}
E_1 = \partial_{u}, \quad E_2= \partial_{v}, \quad E_3= \partial_{w}, \quad F_1 = \partial_{x}, \quad F_2 = \partial_y,\quad F_3= \partial_{z}.
\end{equation}
In particular, the horizontal part of the Jacobi matrix $N^\vers_0(t)$ is simply the Jacobian of the exponential map $(u,v,w) \mapsto \exp^t_{0}(u,v,w)$ computed at $(u_0,v_0,w_{0})$ in these coordinates. A straightforward computation and Lemma~\ref{l:distortioncomputed} yield the following.
\begin{prop}[Heisenberg distortion coefficient]
Let $q \notin \cut(0)$. Then
\begin{equation}\label{eq:coeffheis}
\beta_t(0,q) = 
t \frac{\sin\left(\tfrac{t w_0}{2}\right)}{\sin\left(\tfrac{w_0}{2}\right) }\frac{\sin\left(\tfrac{t w_0}{2}\right) - \tfrac{t w_0}{2} \cos\left(\tfrac{t w_0}{2}\right)}{\sin\left(\tfrac{w_0}{2}\right) - \tfrac{w_0}{2} \cos\left(\tfrac{w_0}{2}\right)}
, \qquad \forall t \in [0,1],
\end{equation}
where $(u_0,v_0,w_0)$ is the initial covector of the unique geodesic joining $0$ with $q$.
\end{prop}

For the Heisenberg group, it is well-known that $t_{\mathrm{cut}}(u_0,v_0,w_{0}) = 2\pi/|w_0|$ (see e.g.\ \cite[Lemma 37]{ABB-Hausdorff}). Hence, since $q \notin \cut(0)$, in the above formula it is understood that $|w_0| < 2\pi$, in which case one can check that $\beta_t(0,q)>0$ for all $t \in (0,1]$.
\begin{rmk}\label{rmk:nodist}
In the above notation, $d^2_{SR}(0,q) = \|\lambda\|^2 = u_0^2 + v_0^2$. We observe that the Heisenberg distortion coefficient \emph{does not} depend on the distance $d_{SR}(0,q)$, but rather on the ``vertical part'' $w_0$ of the covector $\lambda$. See Section~\ref{s:distortionproperties}.
\end{rmk}
The following lemma is a consequence of the inequalities of \cite[Lemma 18]{R-MCPcorank1}.
\begin{lemma}
[Sharp bound to Heisenberg distortion]
\label{l:srhhh}
Let $N  \in \R$. The inequality
\begin{equation}
\beta_t(q_0,q) \geq t^N, \qquad \forall t \in [0,1],
\end{equation}
holds for all points $q_0, q \in \mathbb{H}_3$ with $q \notin \cut(q_0)$, if and only if $N \geq 5$.
\end{lemma}
 One can verify that for $q\in \cut(q_0)$ one has $\beta_t(q_0,q)= +\infty$ for every $t\in (0,1)$. For a  proof see \cite[Cor.~2.1]{BKS-geomheis}. We recover the following  results from \cite{BKS-geomheis}.
\begin{cor}
Let $\mu_0 \in \mathcal{P}^{ac}_c(\mathbb H_3)$, and $\mu_1 \in \mathcal{P}_c(\mathbb H_3)$. Let $\mu_t = (T_t)_\sharp \mu_0 = \rho_t \mathcal{L}^3$ be the unique Wasserstein geodesic joining $\mu_0$ with $\mu_1$. Then,
\begin{equation}
\frac{1}{\rho_t(T_t(x))^{1/3}} \geq \frac{(1-t)^{5/3}}{\rho_0(x)^{1/3}} + \frac{t^{5/3}}{\rho_1(T(x))^{1/3}}, \qquad \mathcal{L}^3-\mathrm{a.e.},\, \forall\, t \in [0,1].
\end{equation}
The above inequality is sharp, in the sense that if one replaces the exponent $5$ with a smaller one, the inequality fails for some choice of $\mu_0,\mu_1$.
\end{cor}
\begin{cor}
For all non-empty Borel sets $A,B \subset \mathbb H_3$, we have 
\begin{equation}
\mathcal{L}^3(Z_t(A,B))^{1/3} \geq (1-t)^{5/3} \mathcal{L}^3(A)^{1/3} + t^{5/3} \mathcal{L}^3(B)^{1/3}, \qquad \forall\, t \in [0,1].
\end{equation}
The above inequality is sharp, in the sense that if one replaces the exponent $5$ with a smaller one, the inequality fails for some choice of $A,B$.
\end{cor}
Notice that, as a consequence of Theorem \ref{t:equivalenza}, we recover also  the following result originally obtained in \cite{Juillet}: the Heisenberg group $\mathbb{H}_3$, equipped with the Lebesgue measure, satisfies the $\mathrm{MCP}(K,N)$ if and only if $N \geq 5$ and $K \leq 0$.


\subsection{Generalized H-type groups}\label{s:genheis}
 These structures were introduced in \cite{BR-MCPHtype}, and constitute a large class of Carnot groups where the optimal synthesis is known. This class contains Kaplan H-type groups, and some of these structures admit non-trivial abnormal minimizing geodesics.

We take the definitions directly from \cite{BR-MCPHtype}, to which we refer for more details. Let $(G,\star)$ be a step $2$ Carnot group, with Lie algebra $\mathfrak{g}$ of rank $k$, dimension $n$ satisfying  $\dim\mathfrak{g}_1=k$, $\dim\mathfrak{g}_2=n-k$ and
\begin{equation}
[\mathfrak{g}_1,\mathfrak{g}_1] = \mathfrak{g}_2, \qquad [\mathfrak{g}_i,\mathfrak{g}_2]=0, \qquad i= 1,2.
\end{equation}
Any choice $g$ of a scalar product on $\mathfrak{g}_1$ induces a left-invariant sub-Riemannian structure $(\distr,g)$ on $G$, such that $\distr(p) = \mathfrak{g}_1(p)$ for all $p \in G$. 
We extend the scalar product $g$ on $\mathfrak{g}_1$ to the whole $\mathfrak{g}$, which we denote with the same symbol. 

For any $V \in  \mathfrak{g}_2$, the skew-symmetric operator $J_V: \mathfrak{g}_1 \to \mathfrak{g}_1$ is defined by
\begin{equation}
g(X,J_V Y) = g(V,[X,Y]), \qquad \forall X,Y \in \mathfrak{g}_1.
\end{equation}
\begin{definition} \label{d:genht}
We say that a step $2$ Carnot group is of \emph{generalized H-type} if there exists a symmetric, non-zero and non-negative operator $S : \mathfrak{g}_1 \to \mathfrak{g}_1$ such that
\begin{equation}\label{eq:13}
J_V J_W + J_W J_V = -2 g(V,W) S^2, \qquad \forall V,W \in \mathfrak{g}_2.
\end{equation}
\end{definition}
\begin{rmk}
The above definition is well posed and does not depend on the choice of the extension of $g$. More precisely, if \eqref{eq:13} is verified for the operators $J_V$ defined by a choice of an extension of $g$, then the operators $\tilde{J}_V$ defined by a different extension $\tilde{g}$ will verify \eqref{eq:13}, with the \emph{same} operator $S$.
\end{rmk}
\begin{rmk}
A generalized H-type group does not admit non-trivial abnormal ge\-o\-des\-ics, and is thus ideal, if and only if $S$ is invertible. When $n=k+1$, we are in the case of corank $1$ Carnot groups. If $S$ is also non-degenerate (and thus $k=2d$ is even and $S>0$), we are in the case of contact Carnot groups. The case $S = \mathrm{Id}_{\mathfrak{g}_1}$ and $k=2d$ corresponds to classical Kaplan H-type groups.
\end{rmk}

The next result follows from the explicit expression for the Jacobian determinant of generalized H-type groups \cite[Lemma 20]{BR-MCPHtype}, which in turn allows to compute explicit distortion coefficients. The latter, in turn, can be bounded by a power law thanks to \cite[Cor.\ 27]{BR-MCPHtype}. In particular, we have the following.
\begin{lemma}
[Sharp bound to generalized H-type distortion]\label{l:srhtype}
Let $(G,\distr,g)$ be a generalized H-type group, with dimension $n$ and rank $k$, equipped with a left-invariant measure $\mis$. Let $N  \in \R$. The inequality
\begin{equation}
\beta_t(x,y) \geq t^N, \qquad \forall t \in [0,1],
\end{equation}
holds for all points $x, y \in G$ with $y \notin \cut(x)$, if and only if $N \geq k+3(n-k)$, the latter number being the geodesic dimension of the Carnot group.
\end{lemma}
As a consequence of \ref{l:srhtype} and Theorem~\ref{t:jacointerpineq} we have the following result, in the ideal setting.
\begin{cor}
\label{c:interpohtype}
Let $(G,\distr,g)$ be an ideal generalized H-type group, with dimension $n$ and rank $k$, equipped with a left-invariant measure $\mis$. Let $\mu_0,\mu_1 \in \mathcal{P}^{ac}_c(G)$. Let $\mu_t = (T_t)_\sharp \mu_0 = \rho_t \mis$ be the unique Wasserstein geodesic joining $\mu_0$ with $\mu_1$. Then,
\begin{equation}\label{eq:jacinterpineq-htype}
\frac{1}{\rho_t(T_t(x))^{1/n}} \geq \frac{(1-t)^{\tfrac{k+3(n-k)}{n}}}{\rho_0(x)^{1/n}} + \frac{t^{\tfrac{k+3(n-k)}{n}}}{\rho_1(T(x))^{1/n}}, \qquad \mis-\mathrm{a.e.},\, \forall\, t \in [0,1].
\end{equation}
The above inequality is sharp, in the sense that if one replaces the exponent $k+3(n-k)$ with a smaller one, the inequality fails for some choice of $\mu_0,\mu_1$.
\end{cor}

\begin{rmk}
The restriction to ideal structures in the above corollary arises from the requirements of the general theory leading to Theorem~\ref{t:jacointerpineq}, while this assumption is not necessary in Lemma~\ref{l:srhtype}. However, we remark that abnormal geodesics of generalized H-type groups are very docile (they consists in straight lines, and never lose minimality). Thus, we expect all the above results to hold also for non-ideal generalized H-type groups. This is supported by the positive results obtained for corank $1$ Carnot groups obtained in \cite{BKS-geomcorank1} and the forthcoming Corollary \ref{c:BMhtype}.
\end{rmk}

The sharp Brunn-Minkowski inequality for ideal generalized H-type groups follows from Theorem~\ref{t:equivalenza} and Lemma~\ref{l:srhtype}. However, thanks to the results of \cite{RY-BMproduct} for product structures, we are able to remove the ideal assumption.
\begin{cor}
\label{c:BMhtype}
Let $(G,\distr,g)$ be a generalized H-type group, with dimension $n$ and rank $k$, equipped with a left-invariant measure $\mis$. For all non-empty Borel sets $A,B \subset G$, we have 
\begin{equation}
\mis(Z_t(A,B))^{\tfrac{1}{n}} \geq (1-t)^{\tfrac{k+3(n-k)}{n}} \mis(A)^{\tfrac{1}{n}} + t^{\tfrac{k+3(n-k)}{n}} \mis(B)^{\tfrac{1}{n}}, \quad \forall\, t \in [0,1].
\end{equation}
The above inequality is sharp, in the sense that if one replaces the exponent $k+3(n-k)$ with a smaller one the inequality fails for some choice of $A,B$.
\end{cor}
\begin{proof}
In this proof, given $N,n \in \N$, we denote $\mathrm{BM}(N,n)$ the following property: for all non-empty Borel sets $A,B$, we have
\begin{equation}
\mis(Z_t(A,B))^{1/n} \geq (1-t)^{N/n} \mis(A)^{1/n} + t^{N/n} \mis(B)^{1/n}, \qquad \forall\, t \in [0,1].
\end{equation}
A generalized H-type group $G$ is the product of an ideal one $G_0$ with dimension $n_0=n-d$ and rank $k_0=k-d$, and $d$ copies of the Euclidean $\R$, for a unique $d \geq 0$ (see \cite[Prop.\ 19]{BR-MCPHtype}). Furthermore, the left-invariant measure $\mis$ of $G$ is the product of the a left-invariant measure $\mis_0$ of $G_0$ and $d$ copies of the Lebesgue measures $\mathcal{L}$ of each factor $\R$. It follows immediately from Lemma~\ref{l:srhtype} and Theorem~\ref{t:equivalenza} that $G_0$, equipped with the measure $\mis_0$, satisfies $\mathrm{BM}(N_0,n_0)$, with $N_0 = k_0+3(n_0-k_0)$. Furthermore, each copy of $\R$, equipped with the Lebesgue measure, satisfies the standard linear Brunn-Minkowski inequality $\mathrm{BM}(1,1)$. It follows from \cite[Thm.\ 3.3]{RY-BMproduct} that the product $G = G_0 \times \R^d$ equipped with the left-invariant measure $\mis = \mis_0 \times \mathcal{L}^d$ satisfies $\mathrm{BM}(N_0+d,n_0+d) = \mathrm{BM}(k+3(n-k),n)$, which is the desired inequality. 

Assume that $G$ satisfies $\mathrm{BM}(k+3(n-k)-\varepsilon,n)$ for some $\varepsilon>0$. Let $x  \notin \cut(y)$. Letting  $A= x \in G$ and $B = \mathcal{B}_r(y) \subset G$, and taking the limit for $r \downarrow 0$, we obtain that $\beta_t(x,y) \geq t^{k+3(n-k) - \varepsilon}$, contradicting the results of Lemma~\ref{l:srhtype}.
\end{proof}

We can also easily recover the following result proved in \cite{BR-MCPHtype}: a generalized H-type group with dimension $n$ and rank $k$, equipped with a left-invariant measure $\mis$, satisfies the $\mathrm{MCP}(K,N)$ if and only if $N \geq k+3(n-k)$ and $K \leq 0$.

\subsection{Grushin plane}\label{s:grush}

The Grushin plane $\G_2$ is the sub-Riemannian structure on $\R^2$ defined by the global generating vector fields
\begin{equation}
X_1 = \partial_x, \qquad X_2 = x \partial_y.
\end{equation}
We stress that the rank of $\distr = \spn\{X_1,X_2\}$ is not constant. More precisely, the structure is Riemannian on $\{ x \neq 0\}$, and it is singular otherwise. We equip the Grushin plane with the Lebesgue measure $\mis= \mathcal{L}^2$ of $\R^2$. In  canonical coordinates $(u,v,x,y)$ on $T^* \R^2$, the corresponding Hamiltonian is 
\begin{equation}
H(u,v,x,y) = \frac{1}{2}(u^2 + x^2 v^2).
\end{equation}
Hamilton equations are easily integrated, and the Hamiltonian flow 
\begin{equation}
e^{t\vec{H}}(u_0,v_0,x_0,y_0) =(u(t),v(t),x(t),y(t))
\end{equation}
with initial covector $\lambda_0 =u_0 dx + v_0 dy  \in T_{(x_0,y_0)}^*\R^2$ reads
\begin{align*}
u(t) & = u_0 \cos(t v_0)- x_0 v_0 \sin(t v_0), \\
v(t) & = v_0, \\
x(t) & = x_0 \cos(t v_0) + u_0 \frac{\sin(t v_0)}{v_0},  \\
y(t) & = y_0 + \frac{\sin (2 t v_0) \left(v_0^2 x_0^2-u_0^2\right)+2 v_0 \left(t \left(v_0^2 x_0^2+u_0^2\right)+u_0 x_0-u_0 x_0 \cos (2 t v_0)\right)}{4 v_0^2}.
\end{align*}
In particular, $\exp_{(x_0,y_0)}^t(u_0,v_0) = (x(t),y(t))$. 
\begin{remark}
Notice that the geodesic flow is an analytic function of the initial data, and if $v_0=0$ the above equations are understood by taking the limit $v_0 \to 0$. We always adopt this convention in this section.
\end{remark}
To compute the distortion coefficients, fix $q_0 =(x_0,y_0) \in \R^2$, let $q \notin \cut(q_0)$, and let $\lambda_0 =u_0 dx + v_0 dy  \in T_{(x_0,y_0)}^*\R^2$ the covector of the unique minimizing geodesic $\gamma:[0,1] \to \R^2$ joining $q_0$ with $q$. 

In order to use Lemma~\ref{l:distortioncomputed} for the computation of the distortion coefficient, we choose the global Darboux frame induced by the global sections of $T(T^*\R^2)$:
\begin{equation}
E_1 = \partial_{u}, \qquad E_2= \partial_{v}, \qquad F_1 = \partial_{x}, \qquad F_2 = \partial_y.
\end{equation}
In particular, the horizontal part of the Jacobi matrix $N^\vers_0(t)$ is simply the Jacobian of the exponential map $(u,v) \mapsto \exp^t_{(x_0,y_0)}(u,v)$ in these coordinates, computed at $(u_0,v_0)$. A straightforward computation and Lemma~\ref{l:distortioncomputed} yield the following.
\begin{prop}[Grushin distortion coefficient]\label{p:coeffgrush}
Let $q \notin \cut(q_0)$. Then
\begin{equation}\label{eq:coeffgrush}
\beta_t(q_0,q) = t\frac{\left(u_0^2+t u_0 v_0^2 x_0+v_0^2 x_0^2\right) \sin (t v_0) -t u_0^2 v_0 \cos (t v_0)}{\left(u_0^2+u_0 v_0^2 x_0+v_0^2 x_0^2\right)\sin (v_0) -u_0^2 v_0 \cos (v_0)}, \qquad \forall t \in [0,1],
\end{equation}
where $(u_0,v_0)$ is the initial covector of the unique geodesic joining $q_0$ with $q$. 
\end{prop}

For the Grushin plane, $t_{\mathrm{cut}}(u_0,v_0) = \pi/|v_0|$ (see \cite[Sec.\ 3.2]{ABS-GaussBonnet} or \cite[Ch.\ 13]{nostrolibro}). Hence, since $q \notin \cut(q_0)$, in the above formula it is understood that $|v_0| < \pi$, in which case one can check directly that $\beta_t(q_0,q)>0$ for all $t \in (0,1]$.

\smallskip
We have the following non-trivial estimate.

\begin{prop}
\label{p:distorsioneg2}
Let $N  \in \R$. The inequality
\begin{equation}
\beta_t(q_0,q) \geq t^N, \qquad \forall t \in [0,1],
\end{equation}
holds for all points $q_0, q \in \mathbb{G}_2$ with $q \notin \cut(q_0)$, if and only if $N \geq 5$.
\end{prop}
\begin{rmk}
Even though the Grushin plane is a quotient of the Heisenberg group, it is not clear how to deduce a bound for distortion coefficients of $\G_2$ starting from the knowledge of the corresponding inequality for $\mathbb{H}_3$. Actually, the most surprising aspect of Proposition~\ref{p:distorsioneg2} is its sharpness. As it is clear from the proof, the necessity of the condition $N\geq 5$ is due to pairs of points $q_0,q$ located on opposite sides of the singular set $\{x=0\}$.
\end{rmk}
\begin{proof}
Let $q = \exp_{q_0}(u_0,v_0)$, with $|v_0|<\pi$, and $q_0 = (x_0,y_0) \in \R^2$. If $x_0 =0$,
\begin{equation}
\beta_t(q_{0},q) = t \times \frac{\sin (t v_0)-t v_0 \cos (t v_0)}{\sin (v_0)-v_0 \cos (v_0)} \geq t^4, \qquad \forall t \in [0,1],
\end{equation}
which follows from the inequality of \cite[Lemma 18]{R-MCPcorank1}, and $|v_0|< \pi$. We now proceed by assuming $x_0 \neq 0$ (by symmetry we actually assume $x_0 > 0$). 

\textbf{Case $v_0 =0$.} This case, corresponding to straight horizontal lines possibly crossing the singular region, is the one which yields the ``only if'' part of the theorem, and we will settle it first. In this case the trigonometric terms disappear, and
\begin{equation}
\beta_t(q_{0},q) = t^2 \times \frac{t^2 u_0^2+3 t u_0 x_0+3 x_0^2}{u_0^2+3 u_0 x_0+3 x_0^2}.
\end{equation}
We want to find the best $N \in \R$, such that for all $x_0 >0$ and $u_0 \in \R$, it holds
\begin{equation}
\frac{t^2 u_0^2+3 t u_0 x_0+3 x_0^2}{u_0^2+3 u_0 x_0+3 x_0^2} \geq t ^{N-2}, \qquad \forall t \in [0,1].
\end{equation}
Since both sides are strictly positive for all $t \in (0,1]$, we can take the logarithms and the above inequality is equivalent to
\begin{equation}
\int_{tu}^u \frac{d}{dz} \log f_{x_0}(z) \, dz \leq (N-2) \int_{tu}^u \frac{d}{dz}\log |z| \, dz, \qquad \forall t \in (0,1),\; u \in \R,
\end{equation}
where $f_{x_0}(z) := z^2 + 3 z x_0 + 3 x_0^2$. This inequality is equivalent to the corresponding inequality for the integrands. After some computations, we obtain the condition
\begin{equation}
(N-4) z^2 + 3x_0(N-3) z + 3x_0^2(N-2) \geq 0, \qquad \forall x_0>0,\, z \in \R.
\end{equation}
One easily checks that the above holds if and only if $N \geq 5$. This proves the ``only if'' part of the statement.

\textbf{Case $v_0 \neq 0$.} By symmetry, we actually assume $v_0 >0$. If $u_0 =0$, then 
\begin{equation}
\beta_t(q_{0},q) = t \frac{\sin(tv_0)}{\sin(v_0)} \geq t^2 \geq t^5, \qquad \forall t \in (0,1).
\end{equation}
Hence in the following we consider $u_0 \neq 0$. We recall the assumptions made so far:
\begin{equation}
x_0 >0, \qquad v_0> 0, \qquad u_0 \neq 0.
\end{equation}
In this case we rewrite \eqref{eq:coeffgrush} as
\begin{equation}
\beta_t(q_0,q) = t \times \frac{f_a(tv_0)}{f_a(v_0)}, \qquad a:= \frac{v_0 x_0}{u_0} \in \R_0 =\R \setminus\{0\},
\end{equation}
where, for all $a \in \R_0$, we defined
\begin{equation}
f_a(\xi):= (1+a \xi + a^2)\sin(\xi) - \xi \cos(\xi).
\end{equation}
It remains to prove that for all $a \in \R_0$ and $N \geq 5$ it holds
\begin{equation}
\frac{f_a(tv_0)}{f_a(v_0)} \geq t^{N-1}, \qquad \forall t \in [0,1].
\end{equation}
In particular, it is sufficient to prove the case $N=5$, which we assume from now on. Since both sides are strictly positive on $t \in (0,1]$, we can take the logarithms and the inequality is equivalent to
\begin{equation}
\int_{tv_0}^{v_0} \frac{d}{d z} \log f_{a}(z) \, dz \leq 4 \int_{tv_0}^{v_0} \frac{d}{dz}\log |z| \, dz, \qquad \forall t \in (0,1),\; a \in \R_0.
\end{equation}
The above inequality is equivalent to the corresponding one for the integrands. After some computation, we obtain the equivalent inequality
\begin{equation}\label{eq:troiaboia}
W_a(z):=Q_a(z) \sin(z) - z P_a(z) \cos(z)\geq 0, \qquad \forall z \in (0,\pi),\; a \in \R_0,
\end{equation}
where we defined the two polynomials:
\begin{equation}
P_a(z) = a (a + z)+4 , \qquad Q_a(z) = (a+z)(4a-z) +4.
\end{equation}
Consider $a \mapsto W_a(z)$. It is easy to check that for all fixed $z \in (0,\pi)$, we have
\begin{equation}
\lim_{a \to \pm \infty} W_a(z) = +\infty.
\end{equation}
Moreover, $\partial_a W_a(z)$ is linear, hence the function $a \mapsto W_a(z)$ has a unique minimum. Then \eqref{eq:troiaboia} is equivalent to the fact that this minimum is non-negative for all $z \in (0,\pi)$. Setting $\partial_a W_a(z) = 0$, we obtain
\begin{equation}
a_{\text{min}} = -\frac{z}{2}\times \frac{3 \sin(z) - z \cos(z)}{4 \sin(z) - z \cos(z)} \leq 0.
\end{equation}
Hence, \eqref{eq:troiaboia} is equivalent to $W_{a_{\text{min}}}(z) \geq 0$ for all $z \in (0,\pi)$. Replacing, and after some computations, we have that such a condition is equivalent to
\begin{equation}\label{eq:troiaboia2}
\bar{W}(z):=\alpha(z) \sin(z)^2 + \beta(z) \cos(z)\sin(z) + \gamma(z) \cos(z)^2 \geq 0, \qquad \forall z \in (0,\pi),
\end{equation}
where we have defined the following polynomials
\begin{align}
\alpha(z) & = (64-25z^2), \\
\beta(z) & = 10z(z^2-8),  \\
\gamma(z) & = z^2 (16-z^2)  >0.
\end{align}


By looking to the graph of $W(z)$, for $z \in (0,\pi)$, one notices that the inequality \eqref{eq:troiaboia2} is extremely sharp for $z$ close to $0$, while it is easier to prove for larger $z$. Hence, we split the proof of \eqref{eq:troiaboia2} into two parts.

\paragraph{(i) Proof of \eqref{eq:troiaboia2} on $(0,2.67)$} Notice that $\bar{W}^{(i)}(0) =0$ for all $i =0,1,2,3,4$, and
\begin{equation}
\bar{W}^{(5)}(z) = 8 \left(2 z^4+8 z^2+3\right) \sin (2 z)+80 z \cos (2 z).
\end{equation}
This is the first $i$-th derivative whose polynomial factors multiplying the trigonometric functions are all non-negative. Furthermore, recall that
\begin{equation}\label{eq:approxsin}
\sin(x) \geq x - \frac{x^3}{6}, \qquad  x \in [0,+\infty),
\end{equation}
\begin{equation}\label{eq:approxcos}
\cos(x) \geq 1 - \frac{x^2}{2} +\frac{x^4}{24}-\frac{x^6}{720}, \qquad  x \in [0,+\infty).
\end{equation}
Hence, using the explicit form of $\bar{W}^{(5)}$, and the fact that the polynomial factors are non-negative, we obtain
\begin{align}\label{eq:polcom}
\bar{W}^{(5)}(z) & \geq 8 \left(2 z^4+8 z^2+3\right) \left(2z - \frac{4 z^3}{3}\right) + 80z \left(1-2 z^2+\frac{2 z^4}{3}-\frac{4 z^6}{45}\right)\\
& = \frac{64}{9} z \left(18-9 z^2 -4 z^6\right), \qquad \forall z \in (0,\pi).
\end{align}
Integrating five times the above inequality on the interval $[0,z]$, and since $\bar{W}^{(i)}(0) =0$ for all $i =0,1,2,3,4$, we obtain
\begin{equation}
\bar{W}(z) \geq z^6 \left(-\frac{4 z^{6}}{13365}-\frac{z^2}{105}+\frac{8}{45}\right), \qquad \forall z \in (0,\pi).
\end{equation}
The term in parenthesis in right hand side of the above is a third order polynomial, for which we can compute explicitly the roots. The first positive root occurs at $z_* \simeq 2.67491$. Hence the inequality  $\bar{W}(z) \geq 0$ is proved for $z \in (0,2.67)$.

\paragraph{(ii) Proof of \eqref{eq:troiaboia2} on $[2.67,\pi)$} On this interval the inequality is easier to verify with rough estimates. Indeed, $\bar{W}(z)$ in \eqref{eq:troiaboia2} is the sum of three terms, and it is sufficient to bound each one of them with the corresponding minimum. The first term
\begin{equation}
(64 - 25 z^2) \sin(z)^2  \quad \text{attains its minimum on $[2.67,\pi)$ at $z = 2.67$},
\end{equation}
where its value is approximately $-23.57$. The second term
\begin{equation}
10 z (z^2-8) \sin(z)\cos(z) \quad \text{attains it minimum on $[2.67,\pi)$ at $z \simeq 3$},
\end{equation}
where its value is approximately $-4.20$. The third term
\begin{equation}
z^2 (16 - z^2) \cos(z)^2 \quad \text{attains its minimum on $[2.67,\pi)$ at $z = 2.67$}.
\end{equation}
where its value is approximately $50.19$. Thus $\bar{W}(z)$ is larger than the sum of the three aforementioned values, which is positive. This proves \eqref{eq:troiaboia2}, and concludes the proof of the proposition.
\end{proof}
Since the Grushin structure is ideal, one obtains the following consequences.
\begin{cor}
\label{c:grushin-interpolation}
Let $\mu_0,\,\mu_1 \in \mathcal{P}^{ac}_c(\G_2)$. Let $\mu_t = (T_t)_\sharp \mu_0 = \rho_t \mathcal{L}^2$ be the unique Wasserstein geodesic joining $\mu_0$ with $\mu_1$. Then,
\begin{equation}\label{eq:jacinterpineq-grushin}
\frac{1}{\rho_t(T_t(x))^{1/2}} \geq \frac{(1-t)^{5/2}}{\rho_0(x)^{1/2}} + \frac{t^{5/2}}{\rho_1(T(x))^{1/2}}, \qquad \mathcal{L}^2-\mathrm{a.e.},\, \forall\, t \in [0,1].
\end{equation}
The above inequality is sharp, in the sense that if one replaces the exponent $5$ with a smaller one, the inequality fails for some choice of $\mu_0,\mu_1$. If $\mu_1$ is not absolutely continuous, an analogous result holds, provided that $t \in [0,1)$, and that in \eqref{eq:jacinterpineq-grushin} the second term on the right hand side is omitted.
\end{cor}
\begin{cor}
\label{c:grushin-BM}
For all non-empty Borel sets $A,B \subset \G_2$, we have 
\begin{equation}
\mathcal{L}^2(Z_t(A,B))^{1/2} \geq (1-t)^{5/2} \mathcal{L}^2(A)^{1/2} + t^{5/2} \mathcal{L}^2(B)^{1/2}, \qquad \forall\, t \in [0,1].
\end{equation}
The above inequality is sharp, in the sense that if one replaces the exponent $5$ with a smaller one, the inequality fails for some choice of $A,B$.
\end{cor}
Finally, by taking $A = y \in \G_2$, and using the fact that the Grushin plane admits a one-parameter group of metric dilations, we obtain the following result.
\begin{cor}
\label{c:grushin-MCP}
The Grushin plane, equipped with the Lebesgue measure, satisfies the $\mathrm{MCP}(K,N)$ if and only if $N \geq 5$ and $K \leq 0$. 
\end{cor}
On the other hand, the Grushin \emph{half}-planes satisfy the $\mathrm{MCP}(K,N)$ if and only if $N \geq 4$ and $K \leq 0$, see \cite{R-gluing}.

\subsection{Sasakian manifolds}\label{s:sasaki}
Let $(\distr,g)$ be a contact sub-Riemannian structure on a $2d+1$-dimensional manifold $M$, that is $\distr =\ker \omega$ where  $\omega \in \Lambda^{1}M$ is a one-form such that  and  $d\omega|_{\distr} $ is non-degenerate. In particular, $(\distr,g)$ is ideal.

The Reeb vector field $X_{0}$ is the unique vector field satisfying $\omega(X_{0})=1$ and $d\omega(X_{0},\cdot)=0$. We  extend the sub-Riemannian metric to a global Riemannian structure (that we denote with the same symbol $g$) by declaring $X_0$ to an unit vector orthogonal to $\distr$. We denote by $\mis$ the corresponding canonical measure.
We define the \emph{contact endomorphism} $J:TM\to TM$ by:
\begin{equation}
g(X,JY)=d\omega(X,Y),\qquad \forall X,Y\in \Gamma(TM). 
\end{equation}
The structure is called \emph{Sasakian} if  the $(1,1)$ tensor defined on $M\times \R$
\begin{equation}
\mathbf{J}(X,f\partial_t) = (JX-fX_0,\omega(X)\partial_t),
\end{equation}
defines a complex structure. We denote with $\mathrm{R}^*$ and $\mathrm{Ric}^*$ the Riemann and Ricci tensors associated with the Tanaka-Webster connection $\nabla^*$ (see references below for details and precise definitions).

In \cite{AAPL,LLZ-Sasakian} it is proved that a $2d+1$-dimensional Sasakian structure endowed with its canonical volume satisfies the $\mathrm{MCP}(0,2d+3)$, under some curvature bounds. 
Thanks to Theorem~\ref{t:equivalenza} and \cite[Thm.~1.1]{LLZ-Sasakian}, we have the following consequence.
\begin{cor}
\label{c:sasaki}
Let $(\distr,g)$ a sub-Riemannian Sasakian structure on a $2d+1$-di\-men\-sion\-al manifold $M$, equipped with its canonical measure $\mis$. 
Assume that
\begin{itemize}
\item[(i)] $\mathrm{R}^*(X,JX,JX,X)\geq 0$ for every $X\in \Gamma(\distr)$,
\item[(ii)] $\mathrm{Ric}^*(X,X)-\mathrm{R}^*(X,JX,JX,X)\geq 0$ for every $X\in \Gamma(\distr)$.
\end{itemize}
Then, for all non-empty Borel sets $A,B \subset M$, we have 
\begin{equation}
\mis(Z_t(A,B))^{\tfrac{1}{2d+1}} \geq (1-t)^{\tfrac{2d+3}{2d+1}} \mis(A)^{\tfrac{1}{2d+1}} + t^{\tfrac{2d+3}{2d+1}} \mis(B)^{\tfrac{1}{2d+1}}, \quad \forall\, t \in [0,1].
\end{equation}
The above inequality is sharp, in the sense that if one replaces the exponent $2d+3$ with a smaller one the inequality fails for some choice of $A,B$.
\end{cor}

\section{Properties of the distortion coefficients}\label{s:distortionproperties}

As we have discussed in Section~\ref{s:esempi}, sub-Riemannian distortion coefficients present major differences with respects to the Riemannian case. In this section, we discuss some of their general properties. Henceforth, let $(\distr,g)$ be a fixed ideal sub-Riemannian structure on $M$, and let $x,y \in M$, with $y \notin \cut(x)$.

\subsection{Dependence on the distance} 

Under the above assumptions, $y = \exp_{x}(\lambda)$ for a unique $\lambda \in T_{x}^*M$ such that $\|\lambda\|  =\sqrt{2H(\lambda)}= d_{SR}(x,y)$. In particular, one can regard the sub-Riemannian distortion coefficients as a one-parameter family of functions depending on the initial covector $\lambda \in T^*M$ of a minimizing geodesic joining a pair of points $(x,y) \in M\times M  \setminus \cut(M)$. Loosely speaking:
\begin{equation}
\beta_t(x,\exp_{x}(\lambda)) = f_t(\lambda).
\end{equation}

The basic Riemannian examples where the $\beta_t$'s are explicit are space forms, where they depend on $\lambda$ \emph{only} through its (dual) norm $\|\lambda\| = d(x,y)$, see \eqref{eq:dist-ref-riem-intro}. As we discussed in Section~\ref{s:heis}, in the simplest sub-Riemannian structure, the Heisenberg group, the dependence on $\lambda$ is fundamentally more complicated, and $\beta_t(x,y)$ is \emph{not} a function of the sub-Riemannian distance between $x$ and $y$. A similar phenomenon occurs in the case of the Grushin plane, treated in Section~\ref{s:grush}.

\subsection{Small time asymptotics: proof of Theorem~\ref{t:asymptotics-intro}} 

For Riemannian structures, it is well known that
\begin{equation}
\beta_t(x,y) \sim C(x,y) t^n,
\end{equation}
with $n = \dim(M)$. This is the reason for the presence of a normalization factor $t^{-n}$ in the standard Riemannian distortion coefficients, which we did not include in Definition~\ref{d:introdist}. In fact, in the genuinely sub-Riemannian case, the asymptotic is remarkably different, as stated in Theorem~\ref{t:asymptotics-intro}. This result follows directly from \cite[Sec.\ 5.6, Sec.\ 6.5]{curvature}. We sketch a proof here for completeness.
\begin{proof}[Proof of Theorem~\ref{t:asymptotics-intro}]
 Let $\Sigma_x := M \setminus \cut(x)$. The function $\mathfrak{f} : M \to \R$ defined by $\mathfrak{f}(y) = \tfrac{1}{2}d^2_{SR}(x,y)$ is smooth on $\Sigma_x$. Furthermore, $d_y \mathfrak{f} \in T_y^*M$ is the final covector of the unique geodesic  joining $x$ with $y$. In particular, define the homothety $\phi_t : \Sigma_x \to M$  of ratio $t \in [0,1]$ and center $x \in M$ by the formula
\begin{equation}
\phi_t(y) : = \pi \circ e^{(t-1)\vec{H}}(d_y \mathfrak{f}), \qquad \forall y \in \Sigma_x.
\end{equation}
For all $\Omega \subset \Sigma_x$ we have $Z_t(x,\Omega) = \phi_t(\Omega)$. Since $\phi_t$ is a local diffeomorphism, and $\Sigma_x$ is open, we have that
\begin{equation}
\beta_t(x,y) \sim  \frac{\det(d_y \phi_t)}{\det(d_y \phi_1)}, \qquad \forall y \in \Sigma_x, \quad t \in [0,1].
\end{equation}
By \cite[Lemma 6.24]{curvature}, there exists $C(x,y)>0$ such that
\begin{equation}
\beta_t(x,y) \sim C(x,y) t^{\mathcal{N}_\lambda},
\end{equation}
where $\mathcal{N}_\lambda \in \mathbb{N} \cup \{+\infty\}$ is defined for all $\lambda \in T_x^*M$ in \cite[Def.\ 5.44]{curvature}.

As a consequence of \cite[Prop.\ 5.46]{curvature}, the function $\lambda \mapsto \mathcal{N}_\lambda$ is constant on an open Zariski set $A_x \subseteq T_x^* M$, where it attains its minimal value. In particular, this implies that the points $y \in \Sigma_x$ where $\mathcal{N}(x,y) > \min \mathcal{N}_\lambda$ has zero measure in $\Sigma_x$. The fact that $\mathcal{N}(x) \geq \dim(M)$ is \cite[Prop.\ 5.49]{curvature}.
\end{proof}

We stress that the set of initial covectors $\lambda \in T_x^*M$ such that $\mathcal{N}_\lambda > \mathcal{N}(x)$ can be non-empty. This is the case for all Carnot groups with Goursat-type distribution and dimension $n \geq 4$, such as the Cartan and the Engel groups. See \cite{Isidro}.    	

\appendix

\section{Conjugate times and optimality: proof of Theorem \ref{t:noconj}}\label{app:B}

The aim of this appendix is to give a self-contained proof of the fact that geodesics not containing abnormal segments lose minimality after their first conjugate point,
following \cite{nostrolibro,S-index}. The main difference with respect to the proof of the analogue statement in the Riemannian setting is that the explicit formula for the second variation of energy is usually expressed in terms of Levi-Civita connection and curvature, which are not available in the sub-Riemannian setting. Hence one has to work with a suitable generalization of the index form on the space of controls. Here, the sub-Riemannian structure is not required to be ideal.

A minimizing geodesic is a horizontal curve $\gamma_u$, parametrized with constant speed, such that its control $u$ is a solution of the constrained minimum problem:
\begin{equation}\label{eq:problema}
\min\{J(v) \mid  v \in \mathcal{U},\, E_x(v) = y\}.
\end{equation}
Here $J : \mathcal{U} \to \R$ is the energy functional and $E_{x}:\mathcal{U} \to M$ is the end-point map based at $x$, where  $\mathcal{U}\subseteq L^{2}([0,1],\R^{m})$, cf.\ Section \ref{s:eplm}.

The Lagrange multipliers rule, in the normal case, gives the first order necessary condition for a control (and the corresponding horizontal curve) to be a minimizer: there exists $\lambda \in T_y^*M$, such that
\begin{equation}\label{eq:multipliers2}
\lambda \circ D_u E_x  =  D_u J.
\end{equation}
Hence a solution of \eqref{eq:problema} is a pair $(u,\lambda)$ satisfying \eqref{eq:multipliers2}. Higher order conditions for the minimality of $\gamma_u$ are given by the second variation of $J$ on the level sets of $E_x$. The second differential of the restriction to the level set is not in general the restriction of the second differential to the tangent space to the level set $T_{u}E^{-1}_x(y)=\ker D_{u}E_{x}$. The following formulas hold (see \cite[Ch.\ 8]{nostrolibro} and \cite[Sec.\ 2.4]{riffordbook}).

\begin{prop}[Second variation of the energy] \label{p:hessiana1} Let $\gamma_{u}:[0,1] \to M$ be a normal geodesic  containing no abnormal segments joining $x$ with $y$ satisfying \eqref{eq:multipliers2} for some $\lambda \in T_y^*M$. Then, we have
\begin{equation} \label{eq:hessiana1}
\Hess_{u}J|_{E^{-1}_x(y)}(v)=D^{2}_{u}J(v)-\lambda \circ D^{2}_{u}E_{x}(v), \qquad \forall\, v\in \ker D_{u}E_{x}.
\end{equation}
Moreover we have
\begin{equation}\label{eq:hessianaancillary}
D^{2}_{u}J(v)=\|v\|_{L^{2}},\quad  D^{2}_{u}E_{x}(v)=\underset{0\leq \tau\leq t\leq 1}{\iint}[(P_{\tau,1})_{*}X_{v(\tau)},(P_{t,1})_{*}X_{v(t)}](y)d\tau dt,
\end{equation}
where $X_{v(t)}:=\sum_{i=1}^{m}v_{i}(t)X_{i}$ and $P_{\tau,t}$ denotes the flow of the non-autonomous vector field $X_{u(t)}$, with initial datum at time $\tau$ and final time $t$.
\end{prop}

Given a pair $(u,\lambda)$ such that $\gamma_{u}$ is a normal geodesic satisfying the first order condition \eqref{eq:multipliers2}, we denote by $u^{s}(t):=su(st)$ the reparametrized control associated with the reparametrized trajectory $\gamma_{u^s}(t)=\gamma_{u}(st)$, both defined for $t\in [0,1]$. The covector $\lambda^s =  s (P_{s,1}^*)\lambda \in T_{\gamma_u(s)}^*M$, is a Lagrange multiplier associated with $u^{s}$.

For normal geodesics containing no abnormal segments (see definition~\ref{d:noabseg}), conjugate points (in the sense of definition \ref{def:conj}) can be characterized by the second variation of the energy, as in the Riemannian case, cf.\ \cite[Ch.\ 8]{nostrolibro}.
\begin{prop}\label{p:a1}
Assume that a normal geodesic $\gamma_{u}:[0,1]\to M$ contains no abnormal segments. Then 
 $\gamma_{u}(s)$ is conjugate to $\gamma_{u}(0)$ if and only if $\Hess_{u^{s}}J|_{E^{-1}_x(\gamma_{u}(s))}$ is a degenerate quadratic form.
\end{prop}

The next two lemmas, proved in \cite[Ch.\ 8]{nostrolibro}, are crucial. For the reader's convenience, we provide a sketch of the proof. 
\begin{lemma}\label{l:a1} 
Assume that a normal geodesic $\gamma_{u}:[0,1]\to M$ contains no abnormal segments. Define the function $\alpha:(0,1]\to \R$ as follows
\begin{equation}\label{eq:quadcomp}
\alpha(s):= \inf \left\{ \|v\|_{L^{2}}^{2}- \lambda_{s}\circ D^{2}_{u^{s}}E_{x}(v)\mid \|v\|_{L^{2}}^{2}=1,\; v\in \ker D_{u^{s}}E_{x}  \right\}.
\end{equation}
Then $\alpha$ is continuous and has the following properties:
\begin{itemize}
\item[(a)] $\alpha(0):=\lim_{s\to 0}\alpha(s)=1$;
\item[(b)] $\alpha(s)=0$ implies that $\Hess_{u^{s}} J\big|_{E_{x}^{-1}(\gamma^{s}_{u}(1))}$ is degenerate;
\item[(c)] $\alpha$ is monotone decreasing;
\item[(d)] if $\alpha(\bar s)=0$ for some $\bar s>0$, then $\alpha(s)<0$ for $s>\bar s$.
\end{itemize}
\end{lemma}
\begin{proof}[Sketch of the proof] Notice that one can write
\begin{equation}
\|v\|_{L^{2}}^{2}- \lambda_{s}\circ D^{2}_{u^{s}}E_{x}(v)=\langle (I-Q_{s})(v)| v\rangle_{L^{2}},
\end{equation}
where $Q_{s}:L^{2}([0,1],\R^{m})\to L^{2}([0,1],\R^{m})$ is a compact and symmetric operator 
As a consequence, one can prove that the infimum in \eqref{eq:quadcomp} is attained. 

Since every restriction $\gamma_{u}|_{[0,s]}$ is not abnormal, the rank of $D_{u^{s}}E_{x}$ is maximal, equal to $n$, for all $s \in (0,1]$. Then, by Riesz representation Theorem, we find a continuous orthonormal basis $\{v_i^s\}_{i \in \N}$ for $ \ker D_{u^{s}}E_{x}$, yielding a continuous one-parameter family of isometries $\phi_s:  \ker D_{u^{s}}E_{x} \to \mathcal{H}$ on a fixed Hilbert space $\mathcal{H}$. Since also $s \mapsto Q_s$ is continuous (in the norm topology), we reduce \eqref{eq:quadcomp} to
\begin{equation}
\alpha(s) =1 - \sup\{ \langle \phi_s \circ Q_s \circ \phi_s^{-1}(w)|w\rangle_{\mathcal{H}} \mid w \in \mathcal{H},\; \|w\|_{\mathcal{H}}=1\},
\end{equation}
where the composition $\tilde{Q}_s:=\phi_s\circ Q_s\circ \phi_s^{-1}$ is a continuous one-parameter family of symmetric and compact operators on a fixed Hilbert space $\mathcal{H}$. The supremum coincides with the largest eigenvalue of $\tilde{Q}_s$, which is well known to be continuous as a function of $s$ if $\tilde{Q}_s$ is (see \cite[V Thm. 4.10]{Kato}). This proves that $\alpha$ is  continuous.

By a rescaling one can see that
\begin{equation}
D^{2}_{u^{s}}E_{x}(v)=s^{2}\underset{0\leq \tau\leq t\leq 1}{\iint}[(P_{s\tau,1})_{*}X_{v(s\tau)},(P_{st,1})_{*}X_{v(st)}]|_{\gamma_{u}(s)}d\tau dt.
\end{equation}
Taking the limit $s\to 0$, one can show that $Q_{s}\to 0$, hence $\tilde{Q}_{s}\to 0$, proving (a).

To prove (b), notice that $\alpha(\bar s)=0$ means that $I-Q_{\bar s}\geq 0$, and that  there exists a sequence $v_{n} \in \ker D_{u^{\bar s}} E_x$ of controls with $\|v_{n}\|_{L^2}=1$ and such that $\|v_{n}\|_{L^{2}}^{2}-\langle Q_{\bar s }(v_{n})| v_{n}\rangle_{L^{2}}\to 0$ for $n\to \infty$.   Up to extraction of a sub-sequence, we have that  $v_{n}$ is weakly convergent to some $\bar v$. By compactness of $Q_{\bar s}$, we deduce that $\langle Q_{\bar s}(\bar v)| \bar v \rangle_{L^{2}}=1$. Since $\|\bar v\|^{2}_{L^{2}}\leq 1$, we have $\langle (I-Q_{\bar s})(\bar v)| \bar v\rangle_{L^{2}}=0$. Being $I-Q_{\bar s}$ a bounded, non-negative symmetric operator, and since $\bar{v} \neq 0$, this implies that $A_{\bar{s}}$ is degenerate. 

To prove (c) let us fix $0\leq s \leq s'\leq 1$ and $v\in \ker D_{u^{s}}E_{x}$. Define
\begin{equation}\notag
\widehat{v}(t):=
\begin{cases}
\sqrt{\dfrac{s'}{s}}\,v\left(\dfrac{s'}{s}t\right), \quad 0\leq t\leq \dfrac{s}{s'},\\[0.4cm]
0, \qquad  \dfrac{s}{s'}<t\leq 1.
\end{cases}
\end{equation}
It follows that $\|\widehat{v}\|^{2}_{L^{2}}=\|v\|^{2}_{L^{2}}$, $\widehat{v}\in \ker D_{u^{s'}}E_{x}$, and $D^{2}_{u^{s}}E_{x}(v)=D^{2}_{u^{s'}}E_{x}(\widehat{v})$. As a consequence, $\alpha(s)\geq \alpha(s')$.

%

To prove (d), assume by contradiction that there exists $s_{1}>\bar s$ such that $\alpha(s_{1})=0$. By monotonicity of point (c), $\alpha(s)=0$ for every $\bar s \leq s \leq s_{1}$. This implies that every point in the image of $\gamma_{u}|_{[\bar s,s_{1}]}$ is conjugate to $\gamma_{u}(0)$. Thanks to Lemma \ref{l:a2}, the segment $\gamma_{u}|_{[\bar s,s_{1}]}$ is abnormal, contradicting the assumption on $\gamma_{ u}$.
\end{proof}

\begin{lemma} \label{l:a2} 
Let $\gamma:[0,1]\to M$ be a normal geodesic that does not contain abnormal segments. Then the set $\mathcal{T}_{c}:=\{t>0 \mid  \gamma(t) \text{ is conjugate to }\gamma(0)\}$ is discrete.
\end{lemma}

\begin{proof}[Sketch of the proof] Let $\lambda(t)$ be a normal extremal associated with the geodesic $\gamma(t)$, satisfying condition (N) of  Theorem \ref{t:utile}. Assume that the  set $\mathcal{T}_{c}$ has an accumulation point $\gamma(\bar t)$. 
The fact that the Hamiltonian is non-negative yields the existence of a segment $\gamma|_{[\bar t,\bar t+\eps]}$ all of whose points are conjugate to $\gamma(0)$. A computation in local coordinates on $T^{*}M$ shows that  $\gamma|_{[\bar t,\bar t+\eps]}$ is an abnormal extremal, namely satisfies characterization (A) of Theorem \ref{t:utile}. 
\end{proof}
We can now prove the following fundamental result.
\begin{theorem} \label{t:a1}
Let $\gamma:[0,1]\to M$ be a normal geodesic that does not contain abnormal segments. Then,
\begin{itemize}
\item[(i)] $t_{c}:=\inf\{t>0 \mid  \gamma(t) \text{ is conjugate to }\gamma(0)\}>0$;
\item[(ii)] For every $s>t_{c}$ the curve $\gamma|_{[0,s]}$ is not a minimizer.
\end{itemize}
\end{theorem}
\begin{proof}
Claim (i) follows directly from Proposition \ref{p:a1} and (a)--(b) of Lemma \ref{l:a1} (or also, independently, from Lemma \ref{l:a2}). Using also  (d) of Lemma \ref{l:a1}, one obtains claim (ii). Indeed, since the Hessian has a negative eigenvalue, we can find a variation joining the same end-points and shorter than the original geodesic, contradicting the minimality assumption.
\end{proof}

By applying Theorem \ref{t:a1} to every restriction $\gamma|_{[s_{1},s_{2}]}$ with $0\leq s_{1}< s_{2}<1$, we obtain Theorem~\ref{t:noconj} stated in Section~\ref{s:prel}.
\section{Positivity: proof of Lemma~\ref{l:keylemma1}}\label{a:cruciale}

\begin{proof}[Proof of Lemma~\ref{l:keylemma1}]
Recall that $E_1(t), \dots,E_n(t),F_1(t),\dots, F_n(t)$ is a fixed Darboux frame along the normal extremal $\lambda :[0,1] \to T^*M$, with initial covector $\lambda(0) = d_x\phi$. For all $s \in [0,1]$, consider the Jacobi matrices $\mathbf{J}^\vers_s(t)$ and $\mathbf{J}^\hors_s(t)$ defined in Section~\ref{s:sjm}, representing the family of Lagrange subspaces
\begin{equation}
e^{(t-s)\vec{H}}_* \spn\{E_1(s),\dots,E_n(s)\}, \qquad \text{and} \qquad  e^{(t-s)\vec{H}}_*  \spn\{F_1(s),\dots,F_n(s)\},
\end{equation}
respectively. Notice that $N^\vers_0(0) = \mathbbold{0}$ and, by the assumption of the Lemma, $N^\vers_0(t)$ is non-degenerate for all $t \in (0,1)$. We define $K(t):=N^\vers_0(t)^{-1}$.

We prove (a) for $t \in (0,1)$. Since no point $\gamma(t)$ is conjugate to $\gamma(0)$ for $t \in (0,1)$, it is sufficient to prove that $\det K(t)>0$ for small $t>0$. By applying Lemma~\ref{l:subspacesriccati} to the Jacobi matrix $\mathbf{J}^\vers_0(t)$, we obtain that $W(t):= N^\vers_0(t)M^\vers_0(t)^{-1}$ is symmetric and satisfies the Riccati equation
\begin{equation}\label{eq:r1}
\dot{W} = B(t) + A(t)^* W + W A(t) + W R(t) W, \qquad W(0) = \mathbbold{0}.
\end{equation}
Equation \eqref{eq:r1} holds provided that $M^{\vers}_0(t)$ is non-degenerate which, since $M^{\vers}_0(0) = \mathbbold{1}$, holds true for sufficiently small $t>0$. Again by Lemma~\ref{l:subspacesriccati}, $B(t) \geq 0$. Hence, a direct application of the matrix Riccati comparison theorem \cite[Appendix A]{BR-comparison} yields that $W(t) \geq 0$ for $t \in [0,\varepsilon]$. Moreover, since $M^\vers_0(0) = \mathbbold{1}$, and $N^\vers_0(t)$ is non-degenerate for $t \in (0,1)$, we have that $W(t) = N^\vers_0(t)M^\vers_0(t)^{-1} >0$ for small $t$. In particular $\det N^\vers_0(t)M^\vers_0(t)^{-1} > 0$, which implies $\det K(t)>0$, yielding (a).

To prove (b) and (c), we introduce a change of basis lemma, and  the $S$ matrix.

\begin{lemma}[Change of basis]\label{l:changeofbasis}
For all $s \in (0,1)$ and $t \in [0,1]$, we have
\begin{equation}\label{eq:change1}
\mathbf{J}_s^\vers(t) = - \mathbf{J}_0^\vers(t) N_0^\vers(s)^{-1} N_0^\hors(s) N_s^\vers(0) + \mathbf{J}_0^\hors(t) N_s^\vers(0).
\end{equation}
Moreover, for the original Jacobi matrix $\mathbf{J}(t)$, we have
\begin{equation}\label{eq:change2}
\mathbf{J}(t) = \mathbf{J}_0^\vers(t) M(0) + \mathbf{J}_0^\hors(t).
\end{equation}
\end{lemma}
\begin{proof}
To prove \eqref{eq:change1}, let $\mathbf{J}_s^\vers(t) = \mathbf{J}_0^\vers(t)C_\vers + \mathbf{J}_0^\hors(t)C_\hors$ for unique $n\times n$ matrices $C_\hors,C_\vers$. These can be computed by evaluating the horizontal component of both sides at times $t=s$ and $t=0$. To prove \eqref{eq:change2}, let $\mathbf{J}(t) = \mathbf{J}_0^\vers(t)D_\vers + \mathbf{J}_0^\hors(t)D_\hors$ for unique $n\times n$ matrices $D_\hors,D_\vers$. The latter can computed by evaluating both the horizontal and vertical components of $\mathbf{J}(t) = \left(\begin{smallmatrix} M(t) \\ N(t) \end{smallmatrix} \right)$ at time $t=0$.
\end{proof}
\begin{lemma}[$S$ matrix]\label{l:smatrix}
Consider the smooth family of $n\times n$ matrices 
\begin{equation}
S(t):=N^\vers_0(t)^{-1} N^\hors_0(t), \qquad \forall\, t \in (0,1).
\end{equation}
Such a matrix is symmetric and $\dot{S}(t) \leq 0$.
\end{lemma}
\begin{proof}
In order to prove the lemma, we start by clarifying the geometric interpretation of $S(t)$. Indeed, observe that, letting
\begin{equation}
Z(t):= E(t)\cdot\left( M^\vers_0(t) S(t) - M^\hors_0(t)\right),
\end{equation}
we have
\begin{equation}\label{eq:tupla}
E(0) \cdot S(t) - F(0) = e^{-t\vec{H}}_* Z(t).
\end{equation}
In particular, $Z(t)$ represents a $n$-tuple of vertical vector fields along $\lambda(t)$, and the left hand side of \eqref{eq:tupla} generates the smooth curve of Lagrange subspaces $\Lambda(t):=e^{-t\vec{H}}_* \ver_{\lambda(t)} \subset T_{\lambda(0)}(T^*M)$. In particular $S(t)$ is symmetric, since\footnote{Here, for $n$-tuples $V$, $W$, the pairing $\sigma(V,W)$ denotes the matrix $\sigma(V_i,W_j)$. Moreover, if $A$ is an $n\times n$ matrix, the notation $W=V \cdot A$ denotes the $n$-tuple $W$ whose $i$-th element is $W_i = \sum_{j=1}^n A_{ji} V_j$.}
\begin{equation}
0 = \sigma\left(E(0) \cdot S(t) - F(0),E(0) \cdot S(t) - F(0)\right) = S(t) - S(t)^*,
\end{equation}
and $S(t)$ is non-increasing:
\begin{equation}
\dot{S}(t)  = \sigma_{\lambda(0)}\left(e^{-t\vec{H}}_* Z(t),\frac{d}{dt} e^{-t\vec{H}}_* Z(t)\right)  = - 2H(Z(t)) \leq 0, \qquad \forall t \in (0,s],
\end{equation}
where in the last equality we identified $Z(t) \in \ver_{\lambda(t)} \simeq T_{\gamma(t)}^*M$. The above inequality holds for any smooth family of vertical vector fields $Z(t)$ along $\lambda(t)$, and follows from a straightforward computation in local coordinates around $\lambda(t)$. 
\end{proof}
Using Lemma~\ref{l:changeofbasis}, one can check that (b) and (c) are equivalent to
\begin{itemize}
\item[$(b')$] $S(t) - S(s) \geq 0$, for all $t \in (0,s]$,
\item[$(c')$] $M(0) + S(s) \geq 0$, for all $t \in (0,1)$.
\end{itemize}
By Lemma~\ref{l:smatrix}, $\dot{S}(t) \leq 0$ for $t \in (0,1)$, thus proving assertion $(b')$. To prove $(c')$,  which a fortiori does not depend on $t$, recall that by the assumptions of Theorem~\ref{t:FR},
\begin{equation}\label{eq:assumptionss}
\frac{1}{2}d^2_{SR}(z,y) + \phi(z) \geq 0, \qquad \forall z \in M,
\end{equation}
with equality at $z=x$. Moreover
\begin{equation}\label{eq:CEMS}
\frac{1}{2s}d^2_{SR}(z,\gamma(s))\geq \frac{1}{2}d^2_{SR}(z,y) -\frac{1-s}{2}d^2_{SR}(x,y) \qquad \forall z \in M, \quad s \in (0,1].
\end{equation}
See \cite[Claim 2.4]{CEMS-interpolations} for a proof of \eqref{eq:CEMS} in the Riemannian setting. The proof carries over to the sub-Riemannian case, and is solely a consequence of the triangular and the arithmetic-geometric inequalities. In particular, a property similar to \eqref{eq:assumptionss} holds replacing $y$ with any midpoint $\gamma(s) \in Z_s(x,y)$, that is
\begin{equation}
\frac{1}{2s}d^2_{SR}(z,\gamma(s)) + \phi (z) \geq \mathrm{const}(s,x,y), \qquad \forall z \in M, \quad s \in (0,1),
\end{equation}
with equality when $z =x$. By Theorem~\ref{t:noconj}, $\gamma(s)$ is not conjugate to $x=\gamma(0)$ along the unique minimizing curve $\gamma$ joining $x$ with $y$, which is not abnormal. Hence $\gamma(s) \notin \cut(x)$, and $c_s(z):= d^2_{SR}(z,\gamma(s))/2s$ is smooth at $z=x$. Furthermore, $\phi$ is two times differentiable at $x$ by the assumptions of Theorem~\ref{t:FR}. Hence, $z\mapsto c_s(z) + \phi(z)$ has a critical point at $z=x$ and a well defined non-negative Hessian
\begin{equation}\label{eq:claimpos}
\Hess(c_s + \phi)|_x \geq 0, 
\end{equation}
as a quadratic form on $T_x M$. We claim that \eqref{eq:claimpos} is equivalent to $(c')$. To prove this claim, we use the next Lemma, which is essentially, a rewording of Lemma~\ref{l:affine}.
\begin{lemma}[Second differential and Hessian]\label{l:secondifferentialofsum}
Let $f,g : M \to \mathbb{R}$, twice differentiable at $x \in M$, and such that $x$ is a critical point for $f+g$. Then
\begin{equation}\label{eq:d2tohess}
d_x^2f - d_x^2(-g) = \Hess(f+g)|_x,
\end{equation}
where we used the fact that the space of second differentials at $\lambda = d_x f = d_x (-g)$ is an affine space on the space of quadratic forms on $T_x M$.
\end{lemma}
\begin{remark}
The difference of second differentials in the left hand side of \eqref{eq:d2tohess} is a linear map $T_x M \to T_\lambda(T^*M)$, with values in $\ver_{\lambda} = T_{\lambda}(T^*_xM) \simeq T_x^*M$, and it is identified with the quadratic form $\Hess(f+g)|_x : T_x M \times T_x M  \to \R$, i.e.\ the Hessian of $f+g$ at the critical point $x$.
 \end{remark}
 
We intend to apply Lemma~\ref{l:secondifferentialofsum} to $\phi + c_s$, which has a minimum point at $x$. Since $d_x(-c_s) = d_x \phi$, both $e^{t\vec{H}}_*\circ d^2_x \phi(X(0))$ and $e^{t\vec{H}}_* \circ d^2_x (-c_s)(X(0))$ are $n$-tuples of Jacobi fields along the same extremal $\lambda(t) = e^{t\vec{H}}(d_x\phi)$. We exploit the relation with Jacobi matrices to compute both second differentials of $\phi$ and $-c_s$ separately.

Since $c_s$ is smooth in a neighbourhood $\mathcal{O}_x$ of $x$, by \cite[Lemma 2.15]{riffordbook}, we have that $\exp_z(s d_z(-c_s)) = \gamma(s)$ for all $z \in \mathcal{O}_x$ and $s \in (0,1)$. Thus,
\begin{equation}\label{eq:verticality}
\pi \circ e^{s\vec{H}}(d_z(- c_s)) = \gamma(s), \quad \forall z \in \mathcal{O}_x \qquad \Rightarrow \qquad e^{s\vec{H}}_* \circ d^2_x (-c_s) (T_x M) = \ver_{\lambda(s)}.
\end{equation}
Equation~\eqref{eq:verticality} implies that the $n$-tuple of Jacobi fields $e^{t\vec{H}}_* \circ d^2_x (-c_s)(X(0))$ is associated with the Jacobi matrix $\mathbf{J}_s^\vers(t)L_s$, for some $n \times n$ matrix $L_s$. Evaluating at $t=0$, we obtain $L_s = N_s^\vers(0)^{-1}$. More precisely, for all $s \in (0,1)$ we have
\begin{equation}\label{eq:fact1}
e^{t\vec{H}}_* \circ d^2_x (-c_s)(X(0)) = E(t) \cdot M^\vers_s(t)N_s^\vers(0)^{-1} + F(t) \cdot N^\vers_s(t)N_s^\vers(0)^{-1}, \quad t \in [0,1].
\end{equation}
Furthermore, by definition of the Jacobi matrix $\mathbf{J}(t)=\left(\begin{smallmatrix} M(t) \\ N(t) \end{smallmatrix}\right)$, we have
\begin{equation}\label{eq:fact2}
e^{t\vec{H}}_* \circ d^2_x \phi (X(0)) = E(t) \cdot M(t) + F(t) \cdot N(t), \qquad t \in [0,1].
\end{equation}
By evaluating \eqref{eq:fact1} and \eqref{eq:fact2} at $t=0$, we obtain
\begin{align}
d^2_x(-c_s)(X(0)) & = E(0) \cdot M^\vers_s(0)N_s^\vers(0)^{-1} + F(0), \\
d^2_x\phi(X(0)) & = E(0) \cdot M(0) + F(0).
\end{align}
In particular, from \eqref{eq:claimpos} and Lemma~\ref{l:secondifferentialofsum} we finally prove $(c')$, since
\begin{align}
0  \leq \Hess(\phi + c_s)|_x & = M(0) - M^\vers_s(0)N_s^\vers(0)^{-1}\\
 & = M(0) + \underbrace{M^\vers_0(0)}_{=\mathbbold{1}} \underbrace{N^\vers_0(s)^{-1} N^\hors_0(s)}_{= S(s)} - \underbrace{M^\hors_0(0)}_{= \mathbbold{0}},
 \end{align}
where, in the second line, we used Lemma~\ref{l:changeofbasis} to eliminate $M_s^\vers(0)$.
\end{proof}
\section{Density formula: proof of Lemma~\ref{l:ambrosio+}}\label{a:ambrosio+}

In the proof of Theorem~\ref{t:abscont} we used a slightly more general reformulation of \cite[Lemma 5.5.3]{AGS-Gradientflows} for non-injective maps. The proof is essentially the same as in the aforementioned reference, and we report it here for completeness.

\begin{proof}[Proof of Lemma~\ref{l:ambrosio+}]
We start by proving that if $\det( \tilde{d}_x f) >0$ for $\mathcal{L}^d$-a.e.\ $x \in \Sigma$, then $f_\sharp(\rho \mathcal{L}^d) \ll \mathcal{L}^d$. For any Borel function $h : \R^d \to [0,+\infty]$, we have the area formula for approximately differentiable maps \cite[Eq.\ 5.5.2]{AGS-Gradientflows}, that is
\begin{equation}\label{eq:areaformula}
\int_{\Sigma_f} h(x)|\det(\tilde{d}_x f)| dx = \int_{\R^d} \sum_{x \in \tilde{f}^{-1}(y) \cap \Sigma_f} h(x) dy.
\end{equation}
Since $f$ is measurable, then $f(x) = \tilde{f}(x)$ up to a $\mathcal{L}^d$-negligible set (see \cite[Remark 5.5.2]{AGS-Gradientflows}), and hence $f_\sharp(\rho \mathcal{L})^d = \tilde{f}_\sharp(\mathcal{L}^d)$. Since $|\det( \tilde{d}_x f)| >0$ for $\mathcal{L}^d$-a.e. $x\in\Sigma\subseteq \Sigma_f$, the function $h: \R^d \to [0,+\infty]$ given by
\begin{equation}
h(x):= \begin{cases} 
\frac{\rho(x) \chi_{\tilde{f}^{-1}(B) \cap \Sigma}(x)}{|\det(\tilde{d}_x f)|} & x \in \Sigma, \\
0 & \text{otherwise},
\end{cases}
\end{equation}
is Borel and well defined. Hence, for any Borel set $B \subset \R^d$, we obtain
\begin{align}
f_\sharp(\rho \mathcal{L}^d)(B) & = (\rho \mathcal{L}^d)(\tilde{f}^{-1}(B) \cap \Sigma) = \int_{\tilde{f}^{-1}(B) \cap \Sigma} \rho(x) \,dx \\
& = \int_{\Sigma_f} \rho(x) \chi_{\tilde{f}^{-1}(B) \cap \Sigma}(x)\, dx \\
& = \int_{\R^d} \sum_{x \in  \tilde{f}^{-1}(y) \cap \Sigma_f}  \frac{\rho(x) \chi_{\tilde{f}^{-1}(B) \cap \Sigma}(x)}{|\det(\tilde{d}_x f)| }\,dy \\
& = \int_{B} \sum_{x \in \tilde{f}^{-1}(y) \cap \Sigma}  \frac{\rho(x)}{|\det(\tilde{d}_x f)| }\,dy,
\end{align}
where in the last line we used \eqref{eq:areaformula}. In particular if $\mathcal{L}^d(B) =0$, then also $f_\sharp(\rho \mathcal{L}^d)(B)$.

The inverse implication is proved by contradiction. Assume that there exists a Borel set $B \subset \Sigma$ with $\mathcal{L}^d(B) >0$ and $ \det(\tilde{d}_x f) = 0$ on $B$. Then, the area formula \eqref{eq:areaformula} with $h = \chi_{B \cap \Sigma}$ yields
\begin{equation}
0 = \int_{\R^d} \sum_{x \in \Sigma_f \cap \tilde{f}^{-1}(y)} \chi_{B\cap \Sigma}(x) \, dy \geq \mathcal{L}^d(\tilde{f}(B)).
\end{equation}
On the other hand, since $f_\sharp (\rho \mathcal{L}^d) = \tilde{f}_\sharp(\rho \mathcal{L}^d)$, we have
\begin{equation}
f_\sharp (\rho \mathcal{L}^d)(\tilde{f}(B)) = \int_{\tilde{f}^{-1}(\tilde{f}(B))} \rho(x) \, dx \geq \int_{B} \rho(x) \, dx = \mathcal{L}^d(B)>0.
\end{equation}
Thus $f_\sharp(\rho \mathcal{L}^d)$ cannot be absolutely continuous w.r.t. $\mathcal{L}^d$.
\end{proof}

\section*{Acknowledgments}
This work was supported by the Grant ANR-15-CE40-0018 of the ANR, and by a public grant as part of the Investissement d'avenir project, reference ANR-11-LABX-0056-LMH, LabEx LMH, in a joint call with the ``FMJH Program Gaspard Monge in optimization and operation research''. This work has been supported by the ANR project ANR-15-IDEX-02.

\medskip

We wish to thank Andrei Agrachev, Luigi Ambrosio, Zoltan Balogh, Ludovic Rifford and Kinga Sipos for stimulating discussions. We also thank the anonymous referees for their comments.

\bibliographystyle{alphaabbrv}
\bibliography{SR-interpolation}

\end{document}